\documentclass[reqno,10pt]{amsart}

\usepackage{amssymb,psfrag,graphicx}
\usepackage[utf8]{inputenc}
\usepackage{enumerate}
\usepackage{color}

\usepackage{marginnote}
\usepackage{mathrsfs}

\numberwithin{equation}{section}

\theoremstyle{plain}
\newtheorem{thm}{Theorem}[section]
\newtheorem{lem}{Lemma}[section]
\newtheorem{prop}{Proposition}[section]

\theoremstyle{definition}
\newtheorem{rmk}{Remark}[section]

\newcommand{\be}{\begin{equation}}
\newcommand{\ee}{\end{equation}}
\newcommand{\bea}{\begin{eqnarray}}
\newcommand{\eea}{\end{eqnarray}}
\newcommand{\beas}{\begin{eqnarray*}}
\newcommand{\eeas}{\end{eqnarray*}}
\newcommand{\rd}{{\text{\rm d}}}
\DeclareMathOperator{\Exp}{e}

\newcommand{\f}{\mathbf{f}}

\newcommand{\bp}{\mathbf{p}}
\newcommand{\bq}{\mathbf{q}}
\newcommand{\bu}{\mathbf{u}}
\newcommand{\bv}{\mathbf{v}}
\newcommand{\bw}{\mathbf{w}}
\newcommand{\bx}{\mathbf{x}}

\newcommand{\mB}{\mathcal{B}}
\newcommand{\mC}{\mathcal{C}}
\newcommand{\mD}{\mathcal{D}}

\newcommand{\mL}{\mathcal{L}}

\newcommand{\mV}{\mathcal{V}}

\newcommand{\bxi}{\boldsymbol{\xi}}
\newcommand{\bta}{\boldsymbol{\eta}}
\newcommand{\bd}{\boldsymbol{\delta}}
\newcommand{\bphi}{\boldsymbol{\varphi}}
\newcommand{\bpsi}{\boldsymbol{\psi}}
\newcommand{\beps}{\boldsymbol{\varepsilon}}
\newcommand{\tvN}{\widetilde{\bv_N}}

\newcommand{\dist}{\textnormal{dist}}

\newcommand{\vn}{\bv_N}

\definecolor{mygreen}{RGB}{0,102,34}

\begin{document}

\title[Uniform in Time Error Estimates for Fully Discrete DA Algorithms]{Uniform in Time Error Estimates for Fully Discrete Numerical Schemes of a  Data Assimilation Algorithm}

\date{May 3, 2018}

\author[H. A. Ibdah]{Hussain A. Ibdah}
\address{\textnormal{(Hussain A. Ibdah)} Department of Mathematics\\
Texas A\&M University\\ College Station, TX 77843, USA.}
\email[H. A. Ibdah] {hibdah@math.tamu.edu}

\author[C. F. Mondaini]{Cecilia F. Mondaini}
\address{\textnormal{(Cecilia F. Mondaini)} Department of Mathematics\\
Texas A\&M University\\ College Station, TX 77843, USA.}
\email[C. F. Mondaini] {cfmondaini@gmail.com}

\author[E. S. Titi]{Edriss S. Titi}
\address{\textnormal{(Edriss S. Titi)} Department of Mathematics\\
Texas A\&M University\\ College Station, TX 77843, USA. {\bf ALSO}, Department of Computer Science and Applied Mathematics, Weizmann Institute of Science, Rehovot 76100, Israel.}
\email[E. Titi] {titi@math.tamu.edu and edriss.titi@weizmann.ac.il}


\begin{abstract}
	We consider fully discrete numerical schemes for a downscaling data assimilation algorithm aimed at approximating the velocity field of the 2D Navier-Stokes equations corresponding to given coarse mesh observational measurements.  The time discretization is done by considering semi- and fully-implicit Euler schemes, and the spatial discretization is based on a spectral Galerkin method. The two fully discrete algorithms are shown to be unconditionally stable, with respect to the size of the time step, number of time steps and the number of Galerkin modes. Moreover, explicit, uniform in time error estimates between the fully discrete solution and the reference solution corresponding to the observational coarse mesh measurements are obtained, in both the $L^2$ and $H^1$ norms. Notably, the two-dimensional Navier-Stokes equations, subject to the no-slip Dirichlet or periodic boundary conditions, are used in this work as a paradigm. The complete analysis that is presented here can be extended to other two- and three-dimensional dissipative systems under the assumption of global existence and uniqueness.
\end{abstract}



\maketitle

\section{Introduction}

Predicting the future state of certain physical and biological systems is crucial in several different contexts such as in meteorology, oceanography, oil reservoir management, neuroscience, medical science, stock market, etc. Most applications deal with a complex physical system, possessing a large number of degrees of freedom. Theoretical models attempt to capture the complex dynamics of such systems, but often can only be derived under simplifying assumptions which limit its ability to represent reality. Observational measurement data can be used to adjust the model towards reality, but it also presents limitations. Usually data is only available on a coarse spatial mesh and, in addition, is commonly contaminated by errors. The field of downscaling data assimilation comprises the set of techniques used for suitably combining the theoretical model with the observed data in order to obtain an accurate prediction of the future state of the system.


Several data assimilation methods have been developed along the years by a growing community of researchers (see, e.g.,  \cite{AschBocquetNodetbook2016,Daleybook1993,HarlimMajdabook2012,Kalnaybook2003,LawStuartZygalakisbook2015,ReichCotterbook2015} and references therein). In this paper, we focus on the \emph{nudging} (or \emph{Newtonian relaxation}) method. The idea consists in adding an extra term to the original model with the purpose of relaxing the coarse scales of the solution of the modified model towards the spatially coarse observations. Some earlier works have implemented this approach in the context of control theory and for models given as ordinary differential equations (ODEs) \cite{Nijmeijer2001,Thau1973}; while others have provided tentative extensions for models given as partial differential equations (PDEs) \cite{Anthes1974,HokeAnthes1976}. A rigorous treatment was given in \cite{AzouaniOlsonTiti2014} (see also \cite{AzouaniTiti2013}), where a general framework was introduced that can be applied to a large class of dissipative PDEs and various types of observables. Indeed, the broad applicability and complete analysis of this framework has been demonstrated in several works \cite{AlbanezNussenzveigTiti2016,BiswasFoiasMondainiTiti2018,BiswasLariosPei2017,BiswasMartinez2017,BlocherMartinezOlson2017,FarhatJohnstonJollyTiti2017,FarhatJollyTiti2015,FarhatLunasinTiti2016a,FarhatLunasinTiti2016b,FarhatLunasinTiti2016c,FarhatLunasinTiti2017,FoiasMondainiTiti2016,GeshoOlsonTiti2016,JollyMartinezTiti2017,LunasinTiti2015,MarkowichTitiTrabelsi2016, MondainiTiti2018} for 2D and 3D dissipative systems that enjoy the global existence, uniqueness, and finite number of asymptotic (in time) determining parameters.  

In order to illustrate the idea introduced in \cite{AzouaniOlsonTiti2014}, let us consider a system modeled by the two-dimensional incompressible Navier-Stokes equations (2D NSE), given on a spatial domain $\Omega \subset \mathbb{R}^2$ and time interval $(0,\infty) \subset \mathbb{R}$ by
\be\label{classicalNSE}
  \partial_t\bu - \nu \Delta\bu + (\bu\cdot\nabla)\bu + \nabla p = \f, \quad \nabla \cdot \bu = 0, \quad (\bx,t) \in \Omega \times (0, \infty),
\ee
where $\bu = \bu(\bx,t)$ and $p = p(\bx,t)$ are the unknowns and denote the velocity vector field and the pressure, respectively; while $\nu > 0$ and $\f = \f(\bx)$ are given and denote the kinematic viscosity parameter and the body forces applied to the fluid per unit mass, respectively. The 2D NSE are used here as a paradigm of a system for which we can provide the complete analysis and explicit estimates, in terms of the physical parameters, without any {\it ad hoc} assumptions on the global existence, uniqueness, and the size of its solutions. Predictions of the future state of the system can be obtained by providing \eqref{classicalNSE} with a suitable initial condition $\bu(0) = \bu_0$ and integrating \eqref{classicalNSE} until the targeted future time. Given coarse-scale measurements, the problem in traditional data assimilation algorithms consists of finding an initial condition that is a good enough approximation of the present state so that an accurate future prediction can be computed.

Contingent to the analytical tools that will be used in this paper to obtain error estimates, we assume that data is assimilated continuously and is free of errors. In this case, the algorithm introduced in \cite{AzouaniOlsonTiti2014} consists in finding a solution to the following approximate system:
\be\label{classicalDA}
	\partial_t\bv - \nu\Delta\bv + (\bv\cdot\nabla)\bv + \nabla \widetilde{p} = \f - \beta  I_h (\bv - \bu), \quad \nabla\cdot\bv=0, \quad (\bx,t) \in \Omega \times (0, \infty),
\ee
where $\bv = \bv(\bx,t)$ is the approximate velocity vector field and $\widetilde{p} = \widetilde{p}(\bx,t)$ is the associated pressure; $\nu$ and $\f$ are the same from \eqref{classicalNSE}; $h > 0$ denotes the spatial resolution of the measurements; $I_h$ is a finite-rank linear interpolant operator in space; and $\beta > 0$ is the relaxation (or nudging) parameter. The purpose of the second term in the right-hand side of \eqref{classicalDA}, called the feedback-control (nudging) term, is to force the coarse spatial scales of $\bv$, represented by $I_h(\bv)$, towards the given spatially coarse observations of $\bu$, represented by $I_h(\bu)$.

In \cite{AzouaniOlsonTiti2014}, the authors prove that under suitable assumptions on $I_h$, $\beta$ and $h$, the solution $\bv$ of \eqref{classicalDA} corresponding to $I_h(\bu)$ and an arbitrary initial data $\bv(0) = \bv_0$ converges exponentially in time to the reference solution $\bu$ of \eqref{classicalNSE}. The key idea behind this result is the fact that, in general, the long-time behavior of dissipative evolution equations is determined by only a finite number of degrees of freedom \cite{CockburnJonesTiti1997,CockburnJonesTiti1995,FoiasProdi1967,FoiasTemam1984,FoiasTiti1991}, which are represented by the coarse mesh part of the solution. Therefore, given measurements $I_h(\bu(\cdot))$ over a long enough time period $[0,T]$, the value $\bv(T)$ can be used as a proper initialization of \eqref{classicalNSE} from which a future prediction can be made.

However, solutions $\bv$ of \eqref{classicalDA} can only be computed, in practice, through finite-dimensional numerical approximations. A natural question is thus to determine the error between a numerical approximation of $\bv$ and the corresponding (infinite-dimensional) reference solution $\bu$ of \eqref{classicalNSE}. In addition to providing efficient quantitative approximation, numerical schemes should also preserve the qualitative dynamical features of the underlying PDEs,  such as dissipation, symmetry, symplectic structure (for certain Hamiltonian systems) and so forth, as it has been advocated in \cite{FoiasJollyKevrekidisTiti1990, FoiasJollyKevrekidisTiti1994, JollyKevrekidisTiti1991} and references therein. In our case, we are concerned with designing efficient numerical schemes which preserve the dissipation property of \eqref{classicalDA}.

 In \cite{MondainiTiti2018}, these questions are addressed for a spatial discretization of \eqref{classicalDA} given by a Galerkin and then Postprocessing Galerkin method; see section \ref{subsecPPGM}, below. Notably, the error estimates obtained in \cite{MondainiTiti2018} are uniform in time, a consequence of the fact that, under the appropriate conditions on $\beta$ and $h$, the feedback-control (nudging) term in \eqref{classicalDA} imposes a stabilizing mechanism by controlling the large scale instabilities caused by the nonlinear term.

Our goal here is to address the same questions, but in the case of a fully (space and time) discrete numerical approximation of \eqref{classicalDA} by taking a time discretization of the Galerkin spatial approximation scheme given in \cite{MondainiTiti2018}. We analyze two types of implicit Euler schemes: fully-implicit and semi-implicit. The difference lies in the way the nonlinear term in \eqref{classicalDA} is discretized (cf. \eqref{semimp} and \eqref{fullimp}, below). We obtain the following results:
\begin{enumerate}[(i)]
	\item Existence and uniqueness of solutions to both time-discrete schemes (Propositions \ref{uniqsemi} and \ref{exsfull}, and Theorem \ref{fulldecay}, below).
	\item\label{resintroii} Stability (i.e. uniform boundedness with respect to the number of Galerkin modes, time step size and number of time steps) in the $(L^2(\Omega))^2$ and $(H^1(\Omega))^2$ norms (Theorems \ref{semibounds} and \ref{fullbounds}, below). Both schemes, fully- and semi-implicit, are unconditionally stable in this sense.
	\item Continuous dependence on the initial data in various norms (Theorems \ref{dependonICsemi} and \ref{fulldecay}, below). In fact, we prove a stronger result: the difference between any two solutions of the numerical schemes corresponding to different initial data converge to zero as the number of time steps increase. This is valid under a smallness assumption on the time step in the semi-implicit case, and unconditionally in the fully-implicit case.
	\item Explicit error estimates (in the $(L^2(\Omega))^2$ and $(H^1(\Omega))^2$ norms) between the solution of each time-discrete scheme and the corresponding continuous in time solution (Theorems \ref{l2convsemi}, \ref{h1errorsemi} and \ref{thmL2convFullImp}, below). Such error estimates are uniform with respect to the number of time steps. Combined with the results from \cite{MondainiTiti2018}, these yield error estimates between each fully discrete numerical approximation of \eqref{classicalDA} and the corresponding reference solution of \eqref{classicalNSE} (Theorems \ref{thmerrorfulldiscSI} and \ref{thmerrorfulldiscFI}, below), which are also uniform with respect to the number of time steps.
\end{enumerate}

The literature is saturated with various discrete in time numerical schemes that are aimed at approximating the solutions to various dissipative PDEs in general, and to \eqref{classicalNSE} in particular. Any such scheme could, in theory, be applied to approximate \eqref{classicalDA}. Therefore, it is difficult to do justice to all of the work that has been done and list it here. Long time stability and finite time error analysis for various numerical schemes associated to \eqref{classicalNSE} was done previously, see for example, \cite{GottliebToneWang2012, Guermond2015, Heister2017, Ju2002, Shen1990, Shen1992b, Shen1992, ToneWirosoetisno2006}. In this paper we consider the simple schemes studied in \cite{Ju2002} and \cite{ToneWirosoetisno2006}. We notice that a similar numerical analysis in the context of control theory (and with a different form of the feedback-control term) was studied in \cite{GunzburgerManservisi2000}.

 We remark that the previously mentioned results, (i)-(iv), of the discretized version of \eqref{classicalDA} are proved by relying heavily on the extra stabilizing mechanism provided by the feedback-control (nudging) term (cf. inequalities \eqref{displ2} and \eqref{disph1}, below). This allows us to avoid the use of (discrete) uniform Gronwall-type inequalities (cf. \cite{Ju2002}), thereby resulting in sharper estimates, or requiring any smallness assumptions on the time step (cf. \cite{ToneWirosoetisno2006}). Furthermore, one would expect only some of the results (i)-(iv) to hold for the schemes as applied to \eqref{classicalNSE} with unstable dynamics. In particular, one would not expect result (iii) to hold when approximating \eqref{classicalNSE} by such schemes.

Lastly, we emphasize that the 2D NSE is considered here only as a paradigm. Similar results can be obtained for other dissipative evolution equations, such as the 3D Navier-Stokes-$\alpha$ model \cite{AlbanezNussenzveigTiti2016}, the 2D B\'enard convection equations \cite{Altaf2017}, and other models considered in \cite{BiswasLariosPei2017}, \cite{FarhatJohnstonJollyTiti2017}-\cite{FarhatLunasinTiti2017} and \cite{MarkowichTitiTrabelsi2016}.

This paper is organized as follows. In section \ref{secPrelim}, we briefly review the necessary material concerning the 2D NSE (subsection \ref{subsecNSE}), the nudging equation \eqref{classicalDA} (subsection \ref{subsecDA}) and its spatial discretization given by a Postprocessing Galerkin method (subsection \ref{subsecPPGM}). In section \ref{secMainRes}, we present the analysis of the semi-implicit (subsection \ref{subsecSI}) and fully-implicit (subsection \ref{subsecFI}) time-discrete schemes. Finally, in the Appendix we present bounds of the Galerkin approximation of a solution to \eqref{classicalDA} and its time derivative, in some high order Sobolev norms. 

\section{Preliminaries}\label{secPrelim}
In this section, we briefly recall the necessary background concerning the two-dimensional incompressible Navier-Stokes equations \eqref{classicalNSE}, the feedback-control (nudging) data assimilation algorithm \eqref{classicalDA} and its spatial discretization given by the Postprocessing Galerkin method. More detailed discussions related to each topic can be found, e.g., in \cite{ConstantinFoiasbook,Temambook1995, TemamDSbook1997, Temambook2001}, \cite{AzouaniOlsonTiti2014} and \cite{GarciaArchillaNovoTiti1998,GarciaArchillaNovoTiti1999,MondainiTiti2018}, respectively.

\subsection{Two-dimensional incompressible Navier-Stokes equations}\label{subsecNSE}

We consider system \eqref{classicalNSE} with either periodic or no-slip Dirichlet boundary conditions. We assume that the forcing term $\f$ is time independent with values in $(L^2\left(\Omega\right))^2$. However, we remark that all the results concerning stability of the discrete schemes associated to \eqref{classicalDA} (Theorems \ref{semibounds}, \ref{dependonICsemi}, \ref{fullbounds} and \ref{fulldecay}, below) are still valid if we assume $\f \in L^{\infty}([0,\infty); (L^2(\Omega))^2)$. On the other hand, we show that, under the hypothesis of time-independent forcing term, one is able to obtain uniform in time strong error estimates (Theorems \ref{l2convsemi}, \ref{h1errorsemi} and \ref{thmL2convFullImp}, below). Time independence can be relaxed further by assuming that the forcing term is time-analytic and bounded in a strip of the complex plane containing the real line. Such assumptions are sufficient because they imply that the solutions to both \eqref{classicalNSE} and \eqref{classicalDA} become time-analytic in such a strip, provided data is assimilated continuously in time and is error-free (see \cite{FoiasJollyLanRupamYangZhang2014}, \cite{FoiasManleyTemam1988} and \cite{foiastemam1979} for more details regarding the analyticity of the solution to \eqref{classicalNSE} and the Appendix for the time analyticity of the solution to \eqref{classicalDA}). This allows us to use analytic tools to bound the solutions and their derivatives in various (high order) Sobolev norms uniformly in time, thereby allowing for uniform in time error estimates.

In what follows, we adopt the notation used in \cite{ConstantinFoiasbook} (see also \cite{Temambook1995, TemamDSbook1997, Temambook2001}). In the case of no-slip boundary conditions, we assume that $\Omega$ is an open, bounded and connected set with a $C^2$ boundary. We denote by $\mathcal{V}$ the set of all smooth, compactly supported, divergence free, two-dimensional vector fields defined on $\Omega$. In the case of periodic boundary conditions, we consider $\Omega=(0,L)\times(0,L)$, for some $L>0$, as the fundamental domain of periodicity. Moreover, we assume that $\f$ is periodic in both spatial directions (with period $L$) and has zero spatial average over $\Omega$ in the latter case. We abuse notation and denote again by $\mathcal{V}$ the set of all divergence free, two-dimensional trigonometric polynomial vector fields with period $L$ in both spatial directions, having zero spatial averages. Also, we denote by $H$ and $V$ the closure of $\mathcal{V}$ in $(L^2(\Omega))^2$ and $(H^1(\Omega))^2$, respectively, regardless of the boundary conditions being considered.

We equip the spaces $H$ and $V$ with the bilinear forms $(\cdot,\cdot)$ and $(\!(\cdot,\cdot)\!)$ defined by
\[ (\bpsi,\bphi ):=\int_{\Omega}\bpsi(\bx)\cdot\bphi(\bx) \rd \bx \quad \forall \bpsi,\bphi\in H,
\]
\[ (\!(\bpsi,\bphi)\!):=\int_{\Omega}\sum_{j=1}^2\frac{\partial\bpsi(\bx)}{\partial x_j}\cdot\frac{\partial\bphi(\bx)}{\partial x_j} \rd \bx \quad \forall \bpsi,\bphi\in V.
\]
It is clear that $(\cdot,\cdot)$ defines an inner product on $H$ and the fact that $(\!(\cdot,\cdot)\!)$ defines an inner product on $V$ follows from the Poincar\'e inequality, given by
\be\label{Poincare}
  \lambda_1^{1/2}\left|\bphi\right|\leq \left\|\bphi\right\|  \quad \forall \bphi\in V,
\ee
where $\lambda_1$ is the first eigenvalue of the Stokes operator, defined in \eqref{defStokesop}, below. We denote the norms induced from $(\cdot,\cdot)$ and $(\!(\cdot,\cdot)\!)$ by $|\cdot|$ and $\|\cdot\|$, respectively. Moreover, we denote by $H'$ and $V'$ the dual spaces of $H$ and $V$, respectively. We identify $H$ with its dual, so that $V\hookrightarrow H\cong H' \hookrightarrow V'$, with the injections being continuous and compact, and each space dense in the following one. Also, we denote by $\langle \cdot, \cdot \rangle$ the natural duality pairing between $V$ and $V'$, i.e., the action of $V'$ on $V$.

For every $R > 0$, we denote by $B_H(R)$ and $B_V(R)$ the closed balls centered at zero and with radius $R$ with respect to the norms in $H$ and $V$, respectively.

Let $P_{\sigma}: (L^2(\Omega))^2 \rightarrow H$ denote the Leray-Helmholtz projector, and let $A: \mD(A) \subset V \to V'$ and $B(\cdot,\cdot): V \times V \to V'$ be the operators defined as the continuous extensions of
\be\label{defStokesop}
 A(\bpsi):=-P_{\sigma}\Delta \bpsi \quad \forall \bpsi\in\mathcal{V}
\ee
and
\be B\left(\bpsi,\bphi\right):=P_{\sigma}\left(\left(\bpsi\cdot\nabla\right)\bphi\right) \quad \forall (\bpsi,\bphi) \in\mV \times \mV,	
\ee
respectively. We recall that $A$ is called the Stokes operator and $\mD(A) = (H^2(\Omega))^2\cap V$. Moreover, $A:\mD(A)\rightarrow H$ is a positive definite and self-adjoint operator with a compact inverse. Therefore, $H$ admits an orthonormal basis of eigenvectors $\{\bw_j\}_{j\in\mathbb{Z^+}}$ of $A$ corresponding to a nondecreasing sequence of positive eigenvalues $\{\lambda_j\}_{j \in \mathbb{Z^+}}$, with $\lambda_j \to \infty$ as $j \to \infty$ (see, for instance, \cite{ConstantinFoiasbook,Temambook1995, TemamDSbook1997, Temambook2001}). For each $N \in \mathbb{Z^+}$, we denote by $P_N$ the orthogonal projector of $H$ onto $\text{span}\{\bw_1, \ldots, \bw_N\} = P_N H$.

The bilinear operator $B$ satisfies the following orthogonality property:
\be\label{propBorthog1}
 \langle B\left(\bu_1,\bu_2\right), \bu_3\rangle = -\langle B\left(\bu_1,\bu_3\right), \bu_2 \rangle \quad \forall \bu_1,\bu_2, \bu_3 \in V,
\ee
which implies, in particular, that
\be\label{propBorthog2}
 \langle B(\bu_1,\bu_2), \bu_2\rangle = 0 \quad \forall \bu_1, \bu_2 \in V.
\ee

Moreover, the following inequalities hold:
\begin{enumerate}[(i)]
\item\label{ineqBii} For every $\bu_1 \in V$, $\bu_2 \in \mathcal{D}(A)$ and $\bu_3 \in H$,
\be\label{estnonlineartermL4L4L2}
 \left|\langle B\left(\bu_1,\bu_2\right),\bu_3\rangle\right| \leq c_2 \left|\bu_1\right|^{1/2} \left\|\bu_1\right\|^{1/2} \left\|\bu_2\right\|^{1/2} \left|A\bu_2\right|^{1/2} \left|\bu_3\right|;
\ee
\item\label{ineqBiii} For every $\bu_1 \in H$, $\bu_2 \in \mD(A)$ and $\bu_3 \in V$,
\be\label{estnonlineartermL2L4L4}
 \left|\langle B\left(\bu_1,\bu_2\right),\bu_3\rangle\right| \leq c_3 \left|\bu_1\right| \left\|\bu_2\right\|^{1/2} \left|A\bu_2\right|^{1/2} \left|\bu_3\right|^{1/2}\left\|\bu_3\right\|^{1/2};
\ee
\item\label{ineqBvi} For every $\bu_1, \bu_3\in \mathcal{D}(A)$ and $\bu_2\in V$, with $\bu_1 \neq 0$,
\be\label{ineqBG}
\left|\langle B\left(\bu_1,\bu_2\right),A\bu_3\rangle\right| \leq c_B \left\|\bu_1\right\| \left\|\bu_2\right\| \left|A\bu_3\right| \left[ 1 + \log\left( \frac{\left|A\bu_1\right|}{\lambda_1^{1/2}\left\|\bu_1\right\|}\right) \right]^{1/2};
\ee
\item\label{ineqBiv} For every $\bu_1, \bu_2, \bu_3 \in V$, with $\bu_3 \neq 0$,
\be\label{ineqTiti1}
\left|\langle B\left(\bu_1,\bu_2\right),\bu_3\rangle\right| \leq c_T \left\|\bu_1\right\| \left\|\bu_2\right\| \left|\bu_3\right| \left[ 1 + \log\left( \frac{\left\|\bu_3\right\|}{\lambda_1^{1/2}\left|\bu_3\right|}\right) \right]^{1/2};
\ee
\item\label{ineqBv} For every $\bu_1\in V$ and $\bu_2, \bu_3 \in \mathcal{D}(A)$, with $\bu_2 \neq 0$,
\be\label{ineqTiti2}
\left|\langle B\left(\bu_1,\bu_2\right),A\bu_3\rangle\right| \leq c_T \left\|\bu_1\right\| \left\|\bu_2\right\| \left|A\bu_3\right| \left[ 1 + \log\left( \frac{\left|A\bu_2\right|}{\lambda_1^{1/2}\left\|\bu_2\right\|}\right) \right]^{1/2};
\ee
\end{enumerate}
where $c_1, c_2, c_3, c_T$ and $c_B$ are (dimensionless) absolute constants. The proofs of inequalities \eqref{estnonlineartermL4L4L2} and \eqref{estnonlineartermL2L4L4} follow by a suitable application of H\"older's inequality, complemented by Sobolev embedding and interpolation theorems when $\bu_1, \bu_2, \bu_3 \in \mV$, and in the general cases by using a density argument and the continuity of $B$ (cf. \cite{ConstantinFoiasbook,Temambook1995, TemamDSbook1997, Temambook2001}). The proof of inequality \eqref{ineqBG} follows by using the Br\'ezis-Gallouet inequality (see \cite{BrezisGallouet1980,FMTT1983}, see also \cite{Titi1987}), while inequalities \eqref{ineqTiti1} and \eqref{ineqTiti2} were proved in \cite{Titi1987}.

In addition, we have the following inequality valid for every $\alpha > 1/2$ and $\bu_1, \bu_2 \in V$ (see \cite[Proposition 6.1]{ConstantinFoiasbook}):
\be\label{ineqBAalpha}
   |A^{-\alpha}(B(\bu,\bu) - B(\bv,\bv))| \leq c_\alpha  |\Omega|^{\alpha - \frac{1}{2}} \|\bu + \bv\| |\bu - \bv|,
\ee
where $|\Omega|$ denotes the area of $\Omega$ and $c_\alpha > 0$ is a constant depending on $\alpha$ via the Sobolev constants from the Sobolev embeddings of $H^{2\alpha}(\mathbb{R}^2)$ into $L^{\infty}(\mathbb{R}^2)$ and of $H^s(\mathbb{R}^2)$ into $L^q(\mathbb{R}^2)$, with $1 > s > (2 - 2 \alpha)$ and $q = 2/(1-s)$. Hence, $c_\alpha \to \infty$ as $\alpha \to \frac{1}{2}^+$.

Given the setting above, we can rewrite system \eqref{classicalNSE} as the following equivalent infinite-dimensional dynamical system:
\be\label{eqNSE}
  \frac{\rd \bu}{\rd t} + \nu A \bu + B(\bu,\bu) = \f,
\ee
where we abuse notation and denote $P_\sigma \f$ simply by $\f$.

It is well-known that, given any $\bu_0 \in H$, there exists a unique solution $\bu$ of \eqref{eqNSE} on $(0, \infty)$ satisfying $\bu(0) = \bu_0$ and
\[
 \bu \in \mC([0, \infty); H) \cap L^2_{\textnormal{loc}}(0, \infty; V), \quad \frac{\rd \bu}{\rd t} \in L^2_{\textnormal{loc}}(0, \infty; V')
\]
(see, e.g., \cite{ConstantinFoiasbook,Temambook1995, TemamDSbook1997, Temambook2001}). Such $\bu$ is called a \emph{weak solution} of \eqref{eqNSE}, and is denoted from now on simply as a solution of \eqref{eqNSE}.

We now recall some uniform bounds satisfied by any solution of \eqref{eqNSE} when complemented with an initial condition $\bu(0) = \bu_0 \in H$. The inequalities in \eqref{boundsu}, below, are classical and can be found in, e.g., \cite{ConstantinFoiasbook,Temambook1995, TemamDSbook1997, Temambook2001}; while the inequality in \eqref{boundderu}, below, was proved in \cite[Appendix]{FoiasManleyTemam1988}.

First, we recall the definition of the Grashof number, given by
\be\label{defGrashof}
 G = \frac{|\f|}{\nu^2 \lambda_1},
\ee
a dimensionless quantity. Recall that when $G$ is small enough equation \eqref{eqNSE} has a unique globally stable  steady state solution and therefore the dynamics becomes trivial. Throughout this paper we assume that $G$ is large enough, in particular that $G \geq 1$, to avoid such triviality.

\begin{prop}\label{propboundsu}
 Let $\bu_0 \in H$ and let $\bu$ be the unique solution of \eqref{eqNSE} on $[0,\infty)$ satisfying $\bu(0) = \bu_0$. Then, there exists $T_0 = T_0(\nu,\lambda_1,G,|\bu_0|)$ such that
 \be\label{boundsu}
	\left|\bu(t)\right|\leq M_0,\quad \left\|\bu(t)\right\| \leq M_1 \quad \forall t \geq T_0
 \ee
 and
 \be\label{boundderu}
  \left\|\frac{\rd \bu}{\rd t}(t) \right\| \leq R_1 \quad \forall t \geq T_0,
 \ee
 where
\be\label{defM0R1}
  M_0 := 2\nu G, \quad  R_1:=c_4\frac{M_1^3 \Lambda}{\nu}, \quad \Lambda:= 1 + \log\left( \frac{M_1}{\nu \lambda_1^{1/2}} \right),
\ee
and, in the case of periodic boundary conditions,
\be\label{PBCbd}
  M_1 = \nu \lambda_1^{1/2} G,
\ee
while in the case of no-slip Dirichlet boundary conditions,
\be\label{Hbd}
  M_1 = c_5 \nu \lambda_1^{1/2} G \Exp^{\frac{G^4}{2}},
\ee
for absolute constants $c_4$ and $c_5$, with $c_5 \geq 1$.
\end{prop}

\begin{rmk}
 Notice that, from the definitions of $G$ in \eqref{defGrashof} and $M_1$ in Proposition \ref{propboundsu}, it follows that
 \be\label{boundf}
   |\f| \leq \nu \lambda_1^{1/2} M_1.
 \ee
 This inequality will be used several times along this paper in order to express some estimates in terms of $M_1$, which in many cases is the dominating term.
\end{rmk}

\subsection{The feedback-control (nudging) data assimilation algorithm and its spectral Galerkin approximation}\label{subsecDA}

We consider system \eqref{classicalDA} equipped with the same boundary conditions as \eqref{classicalNSE}, be it periodic or no-slip Dirichlet. Within the setting introduced in subsection \ref{subsecNSE}, we can rewrite system \eqref{classicalDA} as the following equivalent infinite-dimensional dynamical system:
\be\label{eqDA}
\frac{\rd \bv}{\rd t} + \nu A\bv + B\left(\bv,\bv\right) = \f - \beta P_{\sigma} I_{h}\left(\bv-\bu\right).
\ee

We assume that the linear interpolant $I_h: (H^1(\Omega))^2\rightarrow (L^2(\Omega))^2$ satisfies the following approximation-of-identity-type property:
\be\label{propIh}
  \|\bphi-I_h(\bphi)\|_{(L^2(\Omega))^2} \leq c_0^{1/2} h \|\bphi\|_{(H^1(\Omega))^2} \quad \forall \bphi \in (H^1(\Omega))^2,
\ee
where $c_0 > 0$ is an absolute constant. Notice that for $\bphi\in V$, $\|\bphi\|$ is equivalent to $\|\bphi\|_{(H^1(\Omega))^2}$ and so in this case we abuse notation and simply use $\|\bphi\|$. Examples of such interpolation operators satisfying this property include: the low Fourier-modes projector $P_N$, for some $N \in \mathbb{Z^+}$ with $\lambda_1 N \leq 1/h^2$; and sum of local spatial averages over finite volume elements (see, e.g., \cite{FoiasTiti1991, JonesTiti1992, JonesTiti1993}).

In the following lemma, we list, for convenience, some technical inequalities involving the interpolation operator $I_h$ that are used several times throughout this paper. In particular, inequalities \eqref{displ2} and \eqref{disph1} are the key inequalities that provide the stabilizing mechanism missing from similar discrete in time numerical schemes as applied to \eqref{classicalNSE}.

\begin{lem}
Suppose $I_h$ satisfies \eqref{propIh} and let $\beta>0$ and $h>0$ such that
\be\label{lemprelimcondbetah}
c_0 \beta h^2 \leq \nu.
\ee
Then, the following inequalities hold:
\begin{enumerate}[(i)]
 \item For every $\bphi\in V$,
\be\label{displ2}
-2\beta\left(P_{\sigma}I_h\left(\bphi\right),\bphi\right)\leq\nu\left\|\bphi\right\|^2-\beta\left|\bphi\right|^2;
\ee
\item For every $\bphi\in\mathcal{D}(A)$,
\be\label{disph1}
-2\beta\left(P_{\sigma}I_h\left(\bphi\right),A\bphi\right)\leq\nu\left|A\bphi\right|^2-\beta\left\|\bphi\right\|^2;
\ee
\item For every $\bpsi\in V$, $\bphi\in H$ and $\alpha_0 > 0$,
\be\label{leftovl2}
2\beta\left(P_{\sigma}I_h\left(\bpsi\right),\bphi\right)\leq\frac{2\beta}{\alpha_0}\left|\bphi\right|^2+\alpha_0\beta\left|\bpsi\right|^2+\alpha_0\nu\left\|\bpsi\right\|^2;
\ee
\item For every $\bpsi\in V$, $\bphi\in \mathcal{D}(A)$, $\alpha_0 >0$ and $\alpha_1 > 0$,
\be\label{leftovh1}
2\beta\left(P_{\sigma}I_h\left(\bpsi\right),A\bphi\right)\leq\frac{\beta}{\alpha_0}\left\|\bphi\right\|^2+\frac{\nu}{\alpha_1}\left|A\bphi\right|^2+\beta\left(\alpha_0+\alpha_1\right)\left\|\bpsi\right\|^2.
\ee
\end{enumerate}
\end{lem}
\begin{proof}
By Cauchy-Schwarz inequality and property \eqref{propIh} of $I_h$, we obtain that
\[
2\beta\left(P_{\sigma}\left(\bpsi-I_h\left(\bpsi\right)\right),\bphi\right)\leq2\beta\left|P_{\sigma}\left(\bpsi-I_h\left(\bpsi\right)\right)\right|\left|\bphi\right|\leq 2\beta c_0^{1/2} h\left\|\bpsi\right\|\left|\bphi\right|.
\]
Now, applying Young's inequality and using hypothesis \eqref{lemprelimcondbetah}, yields
\be\label{comp00}
2\beta\left(P_{\sigma}\left(\bpsi-I_h\left(\bpsi\right)\right),\bphi\right)\leq\alpha_0\nu\left\|\bpsi\right\|^2+\frac{\beta}{\alpha_0}\left|\bphi\right|^2 \quad \forall \alpha_0 > 0.
\ee
Similarly, one can show that
\be\label{comp01}
2\beta\left(P_{\sigma}\left(\bpsi-I_h\left(\bpsi\right)\right),A\bphi\right)\leq\alpha_1\beta\left\|\bpsi\right\|^2+\frac{\nu}{\alpha_1}\left|A\bphi\right|^2 \quad \forall \alpha_1 > 0.
\ee
Notice that
\begin{multline}\label{comp01a}
-2\beta\left(P_{\sigma}I_h\left(\bphi\right),\bphi\right)=-2\beta\left(P_{\sigma}\left(\bphi-\bphi+I_h\left(\bphi\right)\right),\bphi\right)\\
= 2\beta\left(P_{\sigma}\left(\bphi-I_h\left(\bphi\right)\right),\bphi\right)-2\beta\left|\bphi\right|^2.
\end{multline}
Thus, \eqref{displ2} follows from \eqref{comp01a} by using \eqref{comp00} with $\bpsi=\bphi$ and $\alpha_0=1$. In order to prove \eqref{leftovl2}, we write
\begin{multline}
2\beta\left(P_{\sigma}I_h\left(\bpsi\right),\bphi\right)=2\beta\left(P_{\sigma}\left(\bpsi-\bpsi+I_h\left(\bpsi\right)\right),\bphi\right)\\
= 2\beta\left(P_{\sigma}\left(\bpsi-I_h\left(\bpsi\right)\right),\bphi\right)+2\beta\left(\bpsi,\bphi\right),
\end{multline}
from which \eqref{leftovl2} follows by using \eqref{comp00} along with Cauchy-Schwarz and Young's inequalities on the last term. Inequalities \eqref{disph1} and \eqref{leftovh1} follow by similar arguments, but using \eqref{comp01} instead of \eqref{comp00}.
\end{proof}

Now, we consider a spectral Galerkin approximation of the solution of \eqref{eqDA}, given by a function $\bv_N: [0, \infty) \to P_N H$ satisfying the following finite-dimensional system of ordinary differential equations:
\be\label{eqDAGalerkin}
   \frac{\rd \bv_N}{\rd t} + \nu A\bv_N + P_N B\left(\bv_N,\bv_N\right) = P_N \f - \beta P_{\sigma} I_h \left(\bv_N - \bu\right),
\ee

In the following proposition, we present some uniform bounds satisfied by the Galerkin approximation $\bv_N$ and its temporal derivative. For this purpose, we need to assume that $\bu$ is a solution of \eqref{eqNSE} on $[0, \infty)$ such that the uniform bound with respect to the norm in $V$ from \eqref{boundsu} is valid for every $t \geq 0$.
The proof is given in the Appendix.

From now on, we denote by $c$ a positive absolute constant that does not depend on any physical parameter and whose value may change from line to line.

\begin{prop}\label{propboundtimederv}
 Let $\bu$ be a solution of \eqref{eqNSE} on $[0,\infty)$ such that $\|\bu(t)\| \leq M_1$ for every $t \geq 0$. Let $\bv_0 \in B_V(M_1)$ and let $\bv_N$ be the solution of \eqref{eqDAGalerkin} on $[0,\infty)$ corresponding to $I_h(\bu)$ and satisfying $\bv_N(0) = P_N\bv_0$, for an arbitrarily fixed $N\in\mathbb{Z^+}$. Assume that $\beta > 0$ and $h > 0$ satisfy:
 \be\label{propvcondbeta}
  \beta\geq \frac{cM_1^2}{\nu}\left[ 1 + \log \left(\frac{M_1}{\nu \lambda_1^{1/2}} \right) \right]
 \ee
 and
 \be\label{propvcondbetah}
 c_0\beta h^2\leq \nu.
 \ee
 Then, there exists $T_1 = T_1(\nu,\lambda_1,G)$ such that the following bounds hold for every $t \geq T_1$ and $N\in\mathbb{Z^+}$:
 \be\label{boundsv}
    \|\bv_N(t) \| \leq 8 M_1, \quad \left\|\frac{\rd \bv_N}{\rd t}(t) \right\| \leq c_7 R_1, \quad
     |A\bv_N (t)|\leq c_8 M_2, \quad \left|A \frac{\rd \bv_N}{\rd t}(t) \right| \leq c_9 R_2,
 \ee
 where $\{c_j\}_{j=7}^{9}$ are positive absolute constants, $M_1$ and $R_1$ are as given in Proposition \ref{propboundsu}, and
 \be\label{defM2R2}
   M_2 := \frac{M_1}{\nu^{1/2}} \left( \frac{M_1 \Lambda^{1/2}}{\nu^{1/2}} + \beta^{1/2} \right), \quad R_2:= \frac{M_1^3 \Lambda}{\nu^{3/2}} \left( \frac{M_1 \Lambda^{1/2}}{\nu^{1/2}} + \beta^{1/2} \right),
 \ee
 with $\Lambda$ as defined in \eqref{defM0R1}.
\end{prop}
\begin{rmk}
In practice, one would like to chose $\beta$ as small as possible, so that the spatial resolution $h$ can be as large as possible. Therefore, replacing $\beta$ by its lower bound, we conclude that the terms inside the parenthesis of \eqref{defM2R2} are of the same order.	
\end{rmk}

\subsection{A Postprocessing of the Galerkin method}\label{subsecPPGM}

In this subsection, we recall a type of postprocessing as applied to the Galerkin approximation given in \eqref{eqDAGalerkin}. The purpose is to obtain a better approximation of $\bv$ than the one given by the Galerkin method. The idea consists in adding to the Galerkin approximation $\bv_N \in P_N H$ an extra term lying in the complement space $(I - P_N)H =: Q_N H$. This extra term represents an approximation of $Q_N \bv \in Q_N H$.

In \cite{MondainiTiti2018}, following ideas from \cite{GarciaArchillaNovoTiti1998,GarciaArchillaNovoTiti1999}, this is done by using the concept of an approximate inertial manifold, particularly the one introduced in \cite{FoiasManleyTemam1988}. In order to obtain an approximation of $\bv$ at a certain time $T > 0$:
\begin{enumerate}[(i)]
 \item\label{PPGMi} Integrate \eqref{eqDAGalerkin} in time, over the time interval $[0,T]$, to obtain $\bv_N$ and compute $\bv_N(T)$;
 \item\label{PPGMii} Obtain $\bq_N$ satisfying $\nu A \bq_N = Q_N[\f - B(\bv_N(T),\bv_N(T))]$;
 \item\label{PPGMiii} Compute the new approximation to $\bv(T)$, and hence to $\bu(T)$, given by $\bv_N(T) + \bq_N$.
\end{enumerate}

The definition of $\bq_N \in Q_N H$ in item \eqref{PPGMii} is inspired by a construction given in \cite{FoiasManleyTemam1988}, where an approximation of $Q_N \bu$, with $\bu$ being a solution of the 2D NSE, is given by
\[
  Q_N \bu \approx \Phi_1(P_N \bu) := (\nu A)^{-1} Q_N [\f - B(P_N \bu, P_N \bu) ].
\]
The graph of the mapping $\Phi_1: P_N H \to Q_N H$ is called an approximate inertial manifold. Its expression is obtained by applying $Q_N$ to the 2D NSE and discarding lower-order terms, namely the time derivative of $Q_N \bu$ and all the nonlinear terms involving $Q_N \bu$.

A key point in the algorithm \eqref{PPGMi}-\eqref{PPGMiii} is the fact that the approximation of $Q_N \bv$, i.e., $\bq_N \in Q_N H$, is only computed at the final time $T$. This is one of the reasons why the Postprocessing Galerkin method is more computationally efficient than other approximation methods, e.g., the Nonlinear Galerkin method (see, e.g., \cite{DevulderMarionTiti1993, FJKT1988, GrahamSteenTiti1993, JollyKevrekidisTiti1990, MargolinTitiWynne2003, MarionTemam1989}).

In \cite{MondainiTiti2018}, it was proved that the error estimate between an approximation of $\bv$ given by the Postprocessing Galerkin method, i.e. $\bv_N + \Phi_1(\bv_N)$, and a true reference solution $\bu$ of \eqref{eqNSE} is better than the one obtained by using the Galerkin method alone. The result concerning this error estimate is recalled in Theorem \ref{thmerrorPPGM}, below. The proof uses estimates of $Q_N \bu$ with respect to the norms in $H$ and $V$, as well as some properties of the mapping $\Phi_1$. These are recalled in the following two propositions that are proven in \cite{FoiasManleyTemam1988} (see also \cite{Titi1990}).

\begin{prop}\label{propq}
Let $\bu_0 \in H$ and let $\bu$ be a solution of \eqref{eqNSE} on $[0, \infty)$ satisfying $\bu(0) = \bu_0$. Then, there exists $T_{0,1} = T_{0,1}(\nu, \lambda_1,|\f|,|\bu_0|) \geq 0$ such that
 \be\label{boundqL2}
|Q_N \bu (t)|\leq C_0 \frac{L_N}{\lambda_{N+1}}, \quad \forall t \geq T_{0,1}, \quad \forall N \in \mathbb{Z^+},
\ee
\be\label{boundqH1}
\|Q_N \bu (t)\| \leq C_1 \frac{L_N}{\lambda_{N+1}^{1/2}}, \quad \forall t \geq T_{0,1}, \quad \forall N \in \mathbb{Z^+},
\ee
where
\be\label{defLN}
L_N = \left[ 1 + \log\left( \frac{\lambda_N}{\lambda_1}\right) \right]^{1/2},
\ee
\be\label{defC0}
C_0 = c\left( \frac{|Q_N \f| + M_1^2}{\nu} \right),
\ee
\be\label{defC1}
C_1 = c \left( \frac{|Q_N \f| + M_1^2}{\nu} + \frac{M_0 M_1^2}{\nu^2} \right),
\ee
and $M_0$ and $M_1$ are as given in Proposition \ref{propboundsu}.
\end{prop}

\begin{prop}\label{propPhi1}
Let $\bu_0 \in H$ and let $\bu$ be a solution of \eqref{eqNSE} on $[0, \infty)$ satisfying $\bu(0) = \bu_0$. Then, there exists $T_{0,2} = T_{0,2}(\nu, \lambda_1,|\f|,|\bu_0|) \geq 0$ such that
 \be\label{distuPhi1L2aa}
 |\Phi_1(P_N \bu(t)) - Q_N \bu(t)| \leq C \frac{L_N}{\lambda_{N+1}^{3/2}} \quad \forall t \geq T_{0,2}, \quad \forall N \in \mathbb{Z^+},
\ee
and
\be\label{distuPhi1H1aa}
 \|\Phi_1(P_N\bu(t)) - Q_N \bu(t)\| \leq C \frac{L_N}{\lambda_{N+1}} \quad \forall t \geq T_{0,2}, \quad \forall N \in \mathbb{Z^+},
 \ee
 where $C$ is a constant depending on $\nu$, $\lambda_1$ and $|\f|$, but independent of $N$.
\end{prop}

Another important property of the mapping $\Phi_1$ is that its restriction to $P_N B_V(R)$, for any $R > 0$, is a Lipschitz mapping with respect to the norms of both $H$ and $V$. More specifically,
\be\label{Phi1LipsH}
 | \Phi_1(\bp_1) - \Phi_1(\bp_2)| \leq l |\bp_1 - \bp_2| \quad \forall \bp_1, \bp_2 \in P_N B_V(R)
\ee
and
\be\label{Phi1LipsV}
  \| \Phi_1(\bp_1) - \Phi_1(\bp_2)\| \leq l \|\bp_1 - \bp_2\| \quad \forall \bp_1, \bp_2 \in P_N B_V(R),
\ee
where $l = C \lambda_{N+1}^{-1/4}$ and $C$ is a constant depending on $\nu, \lambda_1$ and $R$.

In addition to the properties above, the proof of Theorem \ref{thmerrorPPGM} requires additional properties from the interpolant $I_h$, namely:
\be\label{propIh2}
 \|\varphi - I_h(\varphi)\|_{H^{-1}} \leq c_{-1} h |\varphi|, \quad \forall \varphi \in (L^2(\Omega))^2
\ee
and
\be\label{propIh3}
   |I_h(\bq)| \leq \widetilde{c_0} \frac{|\Omega|^{3/4}}{h^2 \lambda_{N+1}^{1/4}} |\bq|, \quad \forall \bq \in Q_N H,
\ee
where $|\Omega|$ denotes the area of $\Omega$, and $c_{-1}$ and $\widetilde{c_0}$ are positive absolute constants. It is not difficult to see that the example of interpolant operator given by a low Fourier modes projector also satisfies properties \eqref{propIh2} and \eqref{propIh3} above. Moreover, one can show that, under periodic boundary conditions, the operator given as sum of local averages over finite volume elements satisfies these additional properties as well (see \cite[Appendix]{MondainiTiti2018}).

\begin{thm}\label{thmerrorPPGM}
Let $\bu$ be a solution of \eqref{eqNSE} on $[0, \infty)$ satisfying the bounds in \eqref{boundsu}, \eqref{boundqL2}, \eqref{boundqH1}, \eqref{distuPhi1L2aa} and \eqref{distuPhi1H1aa} for every $t \geq 0$. Assume that $I_h$ satisfies properties \eqref{propIh}, \eqref{propIh2} and \eqref{propIh3}. Let $\bv_0 \in B_V(M_1)$ and, for each $N \in \mathbb{Z^+}$, let $\bv_N$ be the unique solution of \eqref{eqDAGalerkin} satisfying $\bv_N(0) = P_N \bv_0$. Fix $\alpha \in (1/2,1)$ and assume $\beta,h > 0$ satisfy
\be\label{condbetaerrPPGM}
  \beta \geq \max\left\{ c \frac{M_1^2}{\nu}\left[ 1 + \log \left(\frac{M_1}{\nu\lambda_1^{1/2}}\right)\right] , \left[ c c_\alpha \left( 1 + \frac{1}{1-\alpha}  \right)\frac{|\Omega|^{\alpha - \frac{1}{2}} M_1}{\nu^{\alpha}}\right]^{\frac{1}{1-\alpha}}\right\}
\ee
and
\be\label{condbetaherrPPGM}
   \max\{ c_0, 4 c_{-1} \} \beta  h^2 < \nu,
\ee
where $c_\alpha$, $c_0$ and $c_{-1}$ are the constants from \eqref{ineqBAalpha}, \eqref{propIh} and \eqref{propIh2}, respectively. Then, for every $N \in \mathbb{Z^+}$, there exists $T_2 = T_2(\nu,\lambda_1,|\f|,N) \geq 0$ such that
\be
 \sup_{t \geq T_2}|[\bv_N(t) + \Phi_1(\bv_N(t))] - \bu(t)| \leq C \frac{L_N}{\lambda_{N+1}^{5/4}}
\ee
and
\be
 \sup_{t \geq T_2}\|[\bv_N(t) + \Phi_1(\bv_N(t))] - \bu(t)\| \leq C \frac{L_N}{\lambda_{N+1}^{3/4}},
\ee
where $C$ is a constant depending on $\nu$, $\lambda_1$, $|\f|$ and $1/h^2$, but independent of $N$.
\end{thm}

\begin{rmk}
	The statement of Theorem \ref{thmerrorPPGM} actually differs slightly from the one given in \cite{MondainiTiti2018}, which required, in particular, the number of modes $N$ to be sufficiently large and also more strict conditions on the parameters $\beta$ and $h$. In order to obtain the more general version stated in Theorem \ref{thmerrorPPGM}, above, one proceeds in the following way: first, showing the upper bound of a solution $\bv_N$ of \eqref{eqDAGalerkin} in the $V$ norm by using arguments similar to the ones from the proof of \eqref{boundtildevNV} in the Appendix; secondly, by using the following estimate for the integral with respect to $s \in (t_0,t)$ of the operator norm of $\nu^{\gamma} A^{\gamma}\Exp^{-(t - s) (\nu A P_N + \beta P_N)}$, for any $\gamma \in [1/2, 1) $, whenever it appears (in particular, in the proofs of Theorems 3.5 and 3.10 in \cite{MondainiTiti2018}):
\be
	\int_{t_0}^t \| \nu^{\gamma} A^{\gamma}\Exp^{-(t - s) (\nu A P_N + \beta P_N)} \|_{\mL(P_N H)} \rd s \leq \frac{(\nu \lambda_1)^\gamma}{\beta} + \frac{1}{\beta^{1-\gamma}} \left( \frac{2}{1-\gamma} + 3 \right).
\ee
\end{rmk}

We conclude the section by summarizing the main assumptions that are used several times throughout this paper:
\renewcommand{\theenumi}{{A}\arabic{enumi}}
\begin{enumerate}
	\item\label{A1} $\bu$ is a solution of \eqref{eqNSE} on $[0,\infty)$ satisfying the uniform bounds from \eqref{boundsu} for every $t \geq 0$;
	\item\label{A2} $I_h$ satisfies \eqref{propIh};
	\item\label{A3} $\beta > 0$ and $h > 0$ satisfy conditions \eqref{propvcondbeta} and \eqref{propvcondbetah}, with an appropriate constant $c$ that does not depend on any physical parameter;
	\item\label{A4} $\tau>0$ and $N\in\mathbb{Z^+}$ are arbitrarily fixed.
\end{enumerate}

\section{Main Results}\label{secMainRes}

In this section, we present the analysis concerning the time-discrete approximations of \eqref{eqDAGalerkin}. Subsection \ref{subsecSI} deals with the semi-implicit Euler scheme, while subsection \ref{subsecFI} concerns the fully implicit Euler scheme. We start by stating a version of the discrete Gronwall lemma that will be needed in the subsequent results. The proof follows by a simple induction argument. Throughout this work, we adopt the convention that $0\in\mathbb{N}$, for simplicity.

\begin{lem}\label{discretegronwall}
Let $\{a_k\}_{k\in\mathbb{N}}$ and $\{b_k\}_{k\in\mathbb{N}}$ be sequences of non-negative real numbers satisfying
\be\label{reggr}
(1+\gamma)a_{k+1}\leq a_k+ b_k \quad \forall k\in \{0, 1,\ldots, n\},
\ee
for some $n\in\mathbb{Z^+}$ and $\gamma \in \mathbb{R}$ such that $(1+\gamma)>0$. Then, it follows that
\be\label{grin1}
	a_m \leq  \frac{a_{0}}{(1+\gamma)^{m}}+\sum_{k=0}^{m-1}\frac{b_k}{(1+\gamma)^{m-k}} \quad \forall m\in \{1,\dots,n+1\}.
\ee
In particular, if $\{b_k\}_{k\in\mathbb{N}}\in l^{\infty}(\mathbb{N})$ and $\gamma\neq0$, then
\be\label{grin2}
a_m \leq \frac{a_0}{(1+\gamma)^{m}}+\frac{1}{\gamma} \sup_{k \in \mathbb{N}} \{ b_k\}  \quad \forall m\in \{1,\ldots,n+1\}.
\ee
Moreover, if \eqref{reggr} is valid for every $k\in \mathbb{N}$, then \eqref{grin1} and \eqref{grin2} hold for every $m\in \mathbb{Z^+}$.
\end{lem}

\subsection{Semi-implicit in time scheme}\label{subsecSI}

We consider a sequence of discrete times $t_k = k \tau$, $k \in \mathbb{N}$, with $\tau > 0$ being the time step size. The semi-implicit Euler method applied to \eqref{eqDAGalerkin} consists in finding, for each $k \in \mathbb{N}$, an approximation of $\vn(t_k)$ given by $\vn^k$ satisfying the following scheme

\be\label{semimp}
\frac{\bv_N^{k+1}-\bv_N^k}{\tau}+\nu A\bv_N^{k+1}+P_NB\left(\bv_N^k,\bv_N^{k+1}\right)=P_N\f-\beta P_N P_\sigma I_h\left(\bv_N^{k+1}-\bu\left(t_{k+1}\right)\right).
\ee

First, we prove existence and uniqueness of the initial-value problem associated to \eqref{semimp}.

\begin{prop}\label{uniqsemi}
Let $\bu$ be a solution of \eqref{eqNSE} on $[0, \infty)$ and assume hypotheses \eqref{A2} and (\textbf{A4}). Suppose that $\beta > 0$ and $h > 0$ satisfy $c_0 \beta h^2 \leq \nu$. Then, given  $\bv_{N,0} \in P_N H$, there exists a unique solution $\{\bv_N^k\}_{k \in \mathbb{N}}$ of \eqref{semimp} corresponding to $I_h(\bu)$ and satisfying $\bv_N^0 = \bv_{N,0}$.	
\end{prop}
\begin{proof}
Since \eqref{semimp} is a linear equation in a finite-dimensional space, existence follows immediately once we prove uniqueness. For proving uniqueness, it suffices to show that given $k \in \mathbb{N}$ and $\bv_N^k \in P_N H$, there exists a unique $\bv_N^{k+1} \in P_N H$ satisfying \eqref{semimp}. Suppose, on the contrary, that there exist two solutions of \eqref{semimp}, namely $\bv_{N,1}^{k+1}$ and $\bv_{N,2}^{k+1}$. Then, $\bxi = \bv_{N,1}^{k+1} - \bv_{N,2}^{k+1}$ satisfies
\be\label{semimpxi}
\frac{1}{\tau}\bxi+\nu A\bxi+P_NB(\bv_N^k,\bxi)=-\beta P_N P_\sigma I_h(\bxi).
\ee
Taking the inner product of \eqref{semimpxi} with $\bxi$ in $H$ and using \eqref{displ2} along with the orthogonality property \eqref{propBorthog2}, we obtain that
\[
\frac{1}{\tau}\left|\bxi\right|^2+\frac{\nu}{2}\left\|\bxi\right\|^2\leq 0,
\]
which implies $\bxi = 0$ and thus proves uniqueness.
\end{proof}

Next, we show the results concerning stability of the scheme \eqref{semimp}. The proof of stability in the $H$ norm is a bit simpler due to the orthogonality property \eqref{propBorthog2} of $B$. In order to establish stability in the $V$ norm, we assume that the initial data $\bv_{N,0}$ belongs to $B_V(M_1)$ and proceed via an inductive argument, by exploiting the fact that the feedback-control (nudging) term provides an extra dissipation term. In particular, this allows us to obtain an upper bound which is independent of the nudging parameter $\beta$, a fact that is crucial for the proof of stability of the fully-implicit scheme in subsection \ref{subsecFI}.

\begin{thm}\label{semibounds}
Assume hypotheses \eqref{A1}-\eqref{A4}. Consider $\bv_{N,0} \in P_N H \cap B_V(M_1)$ and denote the unique solution of \eqref{semimp} corresponding to $I_h(\bu)$ and satisfying $\bv_N^0 = \bv_{N,0}$ by $\{\bv_N^k\}_{k\in \mathbb{N}}$. Then, the following inequalities hold for any $n\in\mathbb{N}$:
\be\label{semil2bd}
  \left|\bv_N^n\right|^2\leq \frac{\left|\bv_{N,0}\right|^2}{\left(1+\frac{\tau}{2}\left(\beta+2\nu \lambda_1\right)\right)^n}+\frac{12\left|\f\right|^2}{\beta\left(\beta+2\nu\lambda_1\right)}+\frac{12\beta M_0^2}{\beta+2\nu\lambda_1}+\frac{12\nu M_1^2}{\beta+2\nu\lambda_1},
\ee
\be\label{semih1bd}
  \left\|\bv_N^n\right\|^2\leq \frac{\left\|\bv_{N,0}\right\|^2}{\left(1+\frac{\tau}{4}\left(\beta+\nu\lambda_1\right)\right)^n}+\frac{24\left|\f\right|^2}{\nu\left(\beta+\nu\lambda_1\right)}+\frac{32\beta M_1^2}{\beta+\nu\lambda_1}.
\ee
In particular, using \eqref{boundf}, we have
\be\label{semibdsimple}
\left|\bv_N^n\right|\leq \lambda_1^{-1/2}\left\|\bv_N^n\right\|\leq 6 \lambda_1^{-1/2}M_1.
\ee
\end{thm}
\begin{proof}
We start by proving inequality \eqref{semil2bd}. Taking the inner product of \eqref{semimp} with $2\tau \bv_N^{k+1}$ in $H$, we obtain that
\begin{multline}\label{comp2}
\left|\bv_N^{k+1}\right|^2+\left|\bv_N^{k+1}-\bv_N^k\right|^2-\left|\bv_N^k\right|+2\tau \nu \left\|\bv_N^{k+1}\right\|^2\\
=2\tau\left(\f,\bv_N^{k+1}\right)-2\tau \beta\left(I_h\left(\bv_N^{k+1}-\bu\left(t_{k+1}\right)\right),\bv_N^{k+1}\right),
\end{multline}
where we used the Hilbert space identity
\be
2(a-b,a)= |a|^2 + |a-b|^2 - |b|^2
\ee
and orthogonality property \eqref{propBorthog2} of the bilinear term $B$. We proceed to bound the terms in the right-hand side of \eqref{comp2}. By Cauchy-Schwarz and Young's inequalities, we have
\be\label{comp3}
2\tau\left(\f,\bv_N^{k+1}\right)\leq \frac{6\tau}{\beta}\left|\f\right|^2+\frac{\tau\beta}{6}\left|\bv_N^{k+1}\right|^2.
\ee
For the second term in the right-hand side of \eqref{comp2}, we write
\begin{multline}\label{comp3a}
	-2\tau \beta\left(I_h\left(\bv_N^{k+1}-\bu\left(t_{k+1}\right)\right),\bv_N^{k+1}\right)\\
	= -2\tau \beta\left(I_h\left(\bv_N^{k+1}\right),\bv_N^{k+1}\right)+2\tau \beta\left(I_h\left(\bu\left(t_{k+1}\right)\right),\bv_N^{k+1}\right)
\end{multline}
Applying \eqref{displ2} to the first term in the right-hand side of \eqref{comp3a} and \eqref{leftovl2} with $\alpha_0=6$ to the second, we obtain
\begin{multline}\label{comp4}
-2 \tau \beta \left(I_h\left(\bv_N^{k+1} - \bu(t_{k+1}) \right),\bv_N^{k+1}\right) \leq \tau \nu \left\|\bv_N^{k+1}\right\|^2 - \frac{2\tau\beta}{3}\left|\bv_N^{k+1}\right|^2\\
+ 6 \tau \beta |\bu(t_{k+1})|^2 + 6 \tau \nu \|\bu (t_{k+1})\|^2\\
\leq \tau \nu \left\|\bv_N^{k+1}\right\|^2-\frac{2\tau\beta}{3}\left|\bv_N^{k+1}\right|^2+6\tau\beta M_0^2+6\tau\nu M_1^2,
\end{multline}
where we used the uniform bounds of $\bu$ from Proposition \ref{propboundsu}. Plugging estimates \eqref{comp3} and \eqref{comp4} into \eqref{comp2}, we obtain, after applying Poincar\'e inequality \eqref{Poincare} and dropping the term $\left|\bv_N^{k+1}-\bv_N^{k}\right|^2$,
\be\label{comp4a}
\left(1+\tau\left(\beta/2+\nu\lambda_1\right)\right)\left|\bv_N^{k+1}\right|^2\leq\left|\bv_N^{k}\right|^2+\frac{6\tau}{\beta}\left|\f\right|^2+6\tau\beta M_0^2+6\tau\nu M_1^2, \quad \forall k \geq 0.
\ee
Now, \eqref{semil2bd} follows from \eqref{comp4a} and Lemma \ref{discretegronwall}.

In order to prove inequality \eqref{semih1bd}, we argue by induction. First, notice that \eqref{semih1bd} is trivially true for $n=0$. Now, let $n\in \mathbb{N}$ be fixed and suppose that
\[
\left\|\bv_N^k\right\|^2\leq \frac{\left\|\bv_{N,0}\right\|^2}{\left(1+\frac{\tau}{4}\left(\beta+\nu\lambda_1\right)\right)^k}+\frac{24\left|\f\right|^2}{\nu\left(\beta+\nu\lambda_1\right)}+\frac{32\beta M_1^2}{\beta+\nu\lambda_1}, \quad \forall k \in \{0,1,\dots,n\}.
\]
Using that $\bv_{N,0} \in B_V(M_1)$ and \eqref{boundf}, it follows in particular that
\be\label{comp4b}
 \left\|\bv_N^k\right\|\leq 6 M_1, \quad \forall k \in \{0,1,\ldots,n\}.
\ee
Now, taking the inner product of \eqref{semimp} with $2\tau A\bv_N^{k+1}$ in $H$, we obtain
\begin{multline}\label{comp6}
\left\|\bv_N^{k+1}\right\|^2+\left\|\bv_N^{k+1}-\bv_N^{k}\right\|^2-\left\|\bv_N^{k}\right\|^2+2\tau\nu\left|A\bv_N^{k+1}\right|^2=\\
2\tau\left(\f,A\bv_N^{k+1}\right)-2\tau\beta\left(I_h\left(\bv_N^{k+1}-\bu\left(t_{k+1}\right)\right),A\bv_N^{k+1}\right)-2\tau\left(B\left(\bv_N^k,\bv_N^{k+1}\right),A\bv_N^{k+1}\right).
\end{multline}
The first two terms on the right-hand side of \eqref{comp6} are handled similarly as in \eqref{comp3} and \eqref{comp4}, so that
\be\label{comp7}
2\tau\left(\f,A\bv_N^{k+1}\right)\leq\frac{6\tau}{\nu}\left|\f\right|^2+\frac{\tau\nu}{6}\left|A\bv_N^{k+1}\right|^2
\ee
and
\begin{multline}\label{comp8}
	- 2\tau \beta\left(I_h\left(\bv_N^{k+1}-\bu\left(t_{k+1}\right)\right),A\bv_N^{k+1}\right)\\
	= 2\tau \beta\left(I_h\left(\bu\left(t_{k+1}\right)\right),A\bv_N^{k+1}\right)-2\tau \beta\left(I_h\left(\bv_N^{k+1}\right),A\bv_N^{k+1}\right)\\
\leq\frac{\tau\beta}{2}\left\|\bv_N^{k+1}\right\|^2+\frac{\tau\nu}{6}\left|A\bv_N^{k+1}\right|^2+8\tau\beta M_1^2+\tau\nu\left|A\bv_N^{k+1}\right|^2-\tau\beta\left\|\bv_N^{k+1}\right\|^2\\
=\frac{7\tau\nu}{6}\left|A\bv_N^{k+1}\right|^2-\frac{\tau\beta}{2}\left\|\bv_N^{k+1}\right\|^2+8\tau\beta M_1^2,
\end{multline}
where we used \eqref{disph1} and \eqref{leftovh1} with $\alpha_0=2$ and $\alpha_1=6$. Plugging estimates \eqref{comp7} and \eqref{comp8} into \eqref{comp6}, we obtain, after dropping the term $\left\|\bv_N^{k+1}-\bv_N^{k}\right\|^2$,
\begin{multline}\label{compp9}
\left(1+\frac{\tau\beta}{2}\right)\left\|\bv_N^{k+1}\right\|^2+\frac{2\tau\nu}{3}\left|A\bv_N^{k+1}\right|^2\leq \left\|\bv_N^{k}\right\|^2+\frac{6\tau}{\nu}\left|\f\right|^2+8\tau\beta M_1^2\\
+2\tau\left|\left(B\left(\bv_N^k,\bv_N^{k+1}\right),A\bv_N^{k+1}\right)\right|.
\end{multline}
We now claim that
\be\label{comppp9}
\left(1+\frac{\tau\beta}{4}\right)\left\|\bv_N^{k+1}\right\|^2+\frac{\tau\nu}{4}\left|A\bv_N^{k+1}\right|^2\leq\left\|\bv_N^k\right\|^2+\frac{6\tau}{\nu}\left|\f\right|^2+8\tau\beta M_1^2 \quad \forall k\in\{0,1,\dots,n\}.
\ee
Indeed, if $\bv_N^{k+1}=0$ then \eqref{comppp9} is trivially true. Else, we use \eqref{ineqTiti2} to get
\begin{multline}\label{comp10}
2\tau\left|\left(B\left(\bv_N^k,\bv_N^{k+1}\right),A\bv_N^{k+1}\right)\right| \leq \\
\leq 2\tau c_T \left\|\bv_N^k\right\| \left\|\bv_N^{k+1}\right\| \left|A\bv_N^{k+1}\right| \left[1+\log\left(\frac{\left|A\bv_N^{k+1}\right|}{\lambda_1^{1/2}\left\|\bv_N^{k+1}\right\|}\right)\right]^{1/2}\\
\leq\frac{\tau\nu}{6}\left|A\bv_N^{k+1}\right|^2+\frac{c\tau}{\nu}\left\|\bv_N^k\right\|^2\left\|\bv_N^{k+1}\right\|^2\left[1+\log\left(\frac{\left|A\bv_N^{k+1}\right|^2}{\lambda_1\left\|\bv_N^{k+1}\right\|^2}\right)\right]\\
\leq\frac{\tau\nu}{6}\left|A\bv_N^{k+1}\right|^2+\frac{c\tau}{\nu}M_1^2\left\|\bv_N^{k+1}\right\|^2\left[1+\log\left(\frac{\left|A\bv_N^{k+1}\right|^2}{\lambda_1\left\|\bv_N^{k+1}\right\|^2}\right)\right],
\end{multline}
where we used \eqref{comp4b} in the last step. Plugging \eqref{comp10} into \eqref{compp9} and rearranging some terms, we obtain
\begin{multline}\label{comp11}
	\left(1+\frac{\tau\beta}{4}\right)\left\|\bv_N^{k+1}\right\|^2+\frac{\tau\nu}{4}\left|A\bv_N^{k+1}\right|^2 \\
	+ \frac{\tau}{4}\left\|\bv_N^{k+1}\right\|^2\left\{ \beta + \nu\lambda_1 \left\{ \frac{\left|A\bv_N^{k+1}\right|^2}{\lambda_1\left\|\bv_N^{k+1}\right\|^2} - \frac{cM_1^2}{\nu^2\lambda_1} \left[ 1 + \log\left(\frac{\left|A\bv_N^{k+1}\right|^2}{\lambda_1\left\|\bv_N^{k+1}\right\|^2}\right)\right]\right\}\right\}\\
	\leq\left\|\bv_N^k\right\|^2+\frac{6\tau}{\nu}\left|\f\right|^2+8\tau\beta M_1^2.
\end{multline}
Using that
\be\label{ineqminlog}
\min_{x\geq1} [x-\alpha\left(1+\log(x)\right)] \geq - \alpha\log(\alpha)
\ee
with
\[
	x = \frac{\left|A\bv_N^{k+1}\right|^2}{\lambda_1\left\|\bv_N^{k+1}\right\|^2}\geq1, \,\, \alpha = \frac{cM_1^2}{\nu^2\lambda_1}>0,
\]
we obtain that the fourth term on the left hand side of \eqref{comp11} is bounded from below by
\[
\frac{\tau}{4}\left\|\bv_N^{k+1}\right\|^2\left\{ \beta - \frac{cM_1^2}{\nu}\log\left(\frac{cM_1^2}{\nu^2\lambda_1}\right)\right\} \geq \frac{\tau}{4}\left\|\bv_N^{k+1}\right\|^2 \left\{ \beta - \frac{cM_1^2}{\nu}\left[ 1 + \log\left(\frac{M_1}{\nu\lambda_1^{1/2}}\right)\right]\right\},
\]
which is non-negative by assumption \eqref{propvcondbeta}, with a suitable absolute constant $c$. This proves \eqref{comppp9}.

Now, applying Poincar\'e inequality, \eqref{Poincare}, to the second term on the left-hand side of \eqref{comppp9} and using Lemma \ref{discretegronwall}, we obtain that
\[
\left\|\bv_N^{n+1}\right\|^2\leq \frac{\left\|\bv_{N,0}\right\|^2}{\left(1+\frac{\tau}{4}\left(\beta+\nu\lambda_1\right)\right)^{n+1}}+\frac{24\left|\f\right|^2}{\nu\left(\beta+\nu\lambda_1\right)}+\frac{32\beta M_1^2}{\beta+\nu\lambda_1},
\]
thereby closing the inductive argument.
\end{proof}

\begin{rmk}
	Notice that hypothesis \eqref{propvcondbeta} on $\beta$ is only needed in the proof of estimate \eqref{semih1bd}, but not \eqref{semil2bd}.\	
\end{rmk}

In the following theorem, we show that solutions of \eqref{semimp} depend continuously on the initial data.

\begin{thm}\label{dependonICsemi}
	Assume the hypotheses of Theorem \ref{semibounds} and suppose further that $\tau$ is chosen to satisfy $\tau \beta \leq 1$. Let $\bv_{N,0}, \overline{\bv}_{N,0} \in P_N H \cap B_V(M_1)$ be two different initial data. Let $\{\vn^{k}\}_{k\in\mathbb{N}}$ and $\{\overline{\bv}_N^k\}_{k\in \mathbb{N}}$ be the unique solutions of \eqref{semimp} corresponding to $I_h(\bu)$ and with initial conditions $\bv_{N,0}$ and $\overline{\bv}_{N,0}$, respectively.  Then,
	\be\label{l2decaysemi}
	\left|\overline{\bv}_N^n - \vn^n\right|^2 \leq \frac{\left|\overline{\bv}_{N,0} - \bv_{N,0} \right|^2}{\left(1+\frac{\tau}{4}\left(\beta+\nu\lambda_1\right)\right)^n}, \quad \forall n \in \mathbb{N},
	\ee
	and, consequently, $\lim_{n\rightarrow\infty} \left|\overline{\bv}_N^n - \vn^n\right|=0$.
\end{thm}
\begin{proof}
	First, notice that, using \eqref{semibdsimple} and similar arguments to the ones used in the proof of inequality \eqref{comppp9}, we can prove now that \eqref{comppp9} is valid for every $k \in \mathbb{N}$, i.e.,
	\be\label{comp30b}
	\left(1+\frac{\tau\beta}{4}\right)\left\|\bv_N^{k+1}\right\|^2+\frac{\tau\nu}{4}\left|A\bv_N^{k+1}\right|^2\leq\left\|\bv_N^k\right\|^2+\frac{6\tau}{\nu}\left|\f\right|^2+8\tau\beta M_1^2, \quad \forall k \in \mathbb{N}.
	\ee
	Dividing \eqref{comp30b} by $(1+\tau\beta/4)$, neglecting $\left\|\bv_N^{n+1}\right\|^2$ and using \eqref{boundf}, we obtain
	\be\label{abd1}
	\frac{\tau\nu}{4+\tau\beta}\left|A\bv_N^{k+1}\right|^2\leq\frac{\left\|\bv_N^n\right\|^2}{1+\frac{\tau\beta}{4}}+\frac{6\tau\nu\lambda_1}{1+\frac{\tau\beta}{4}}M_1^2+\frac{8\tau\beta M_1^2}{1+\frac{\tau\beta}{4}}.
	\ee	
	Using condition \eqref{propvcondbeta} on $\beta$ with an appropriate absolute constant $c$, we see that the second term on the right-hand side of \eqref{abd1} is bounded from above by $\nu^2\lambda_1\leq M_1^2$. Clearly, the last term is bounded by $32 M_1^2$. Moreover, using \eqref{semibdsimple} for estimating the first term, we obtain
	\be\label{comp30c}
	\frac{\tau\nu}{4+\tau\beta}\left|A\bv_N^{k+1}\right|^2\leq 69 M_1^2.
	\ee
	Multiplying \eqref{comp30c} by $5/\nu$ and using the hypothesis $\tau \beta \leq 1$, it follows that
	\be\label{comp30d}
	\tau \left|A\bv_N^{k+1}\right|^2 \leq 345 \frac{M_1^2}{\nu} \leq \beta, \quad \forall k \geq 0,
	\ee
	where in the last inequality we used condition \eqref{propvcondbeta} with a suitable absolute constant $c$.
	
	Now, from \eqref{semimp}, it follows that $\beps^k:= \overline{\bv}_N^k - \vn^k$ satisfies
	\be\label{diffsemi}
	\frac{\beps^{k+1}-\beps^k}{\tau}	+\nu A\beps^{k+1} + P_N B\left(\overline{\bv}_N^{k},\beps^{k+1}\right) + P_N B\left(\beps^{k},\bv_N^{k+1}\right)
	=-\beta P_N P_{\sigma}I_h\left(\beps^{k+1}\right).
	\ee
	Taking the inner product of \eqref{diffsemi} with $2\tau\beps^{k+1}$ in $H$, using the orthogonality property \eqref{propBorthog2} of $B$ and inequality \eqref{displ2}, we obtain that
	\begin{multline}\label{dec1}
		\left(1+\tau\beta\right)\left|\beps^{k+1}\right|^2+\left|\beps^{	k+1}-\beps^{k}\right|^2-\left|\beps^k\right|^2+\tau\nu\left\|\beps^{k+1}\right\|^2\leq 2\tau\left|\left(B\left(\beps^{k},\bv_N^{k+1}\right),\beps^{k+1}\right)\right|\\
		\leq 2\tau\left|\left(B\left(\beps^{k}-\beps^{k+1},\bv_N^{k+1}\right),\beps^{k+1}\right)\right|+2\tau\left|\left(B\left(\beps^{k+1},\bv_N^{k+1}\right),\beps^{k+1}\right)\right|.
	\end{multline}
	Using \eqref{estnonlineartermL2L4L4} and Young's inequality to estimate the first term on the right-hand side of \eqref{dec1}, we have
	\begin{multline*}
		2\tau\left|\left(B\left(\beps^{k}-\beps^{k+1},\bv_N^{k+1}\right),\beps^{k+1}\right)\right|\\
		\leq c\tau\left|\beps^{	k+1}-\beps^{k}\right|\left\|\bv_N^{k+1}\right\|^{1/2}\left|A\bv_N^{k+1}\right|^{1/2}\left|\beps^{k+1}\right|^{1/2}\left\|\beps^{k+1}\right\|^{1/2}\\
		\leq \left|\beps^{	k+1}-\beps^{k}\right|^2+c\tau^2\left\|\bv_N^{k+1}\right\|\left|A\bv_N^{k+1}\right|\left|\beps^{k+1}\right|\left\|\beps^{k+1}\right\|.
	\end{multline*}
	From \eqref{comp30d} and condition $\tau \beta \leq 1$, it follows that $\tau\left|A\bv_N^{k+1}\right|\leq1$. Using this along with \eqref{semibdsimple}, yields
	\begin{multline}\label{dec2}
		2\tau\left|\left(B\left(\beps^{k}-\beps^{k+1},\bv_N^{k+1}\right),\beps^{k+1}\right)\right|\leq \left|\beps^{	k+1}-\beps^{k}\right|^2+c\tau M_1\left|\beps^{k+1}\right|\left\|\beps^{k+1}\right\|\\
		\leq \left|\beps^{	k+1}-\beps^{k}\right|^2+\frac{\tau\nu}{4}\left\|\beps^{k+1}\right\|^2+\frac{c\tau M_1^2}{\nu}\left|\beps^{k+1}\right|^2.
	\end{multline}
	For the second term in the right-hand side of \eqref{dec1}, we use \eqref{ineqTiti1}, Young's inequality and \eqref{semibdsimple} to obtain
	\begin{multline}\label{dec3}
		2\tau\left|\left(B\left(\beps^{k+1},\bv_N^{k+1}\right),\beps^{k+1}\right)\right|\\
		\leq \frac{\tau\nu}{4}\left\|\beps^{k+1}\right\|^2+\frac{c\tau M_1^2}{\nu}\left|\beps^{k+1}\right|^2\left[1+\log\left(\frac{\left\|\beps^{k+1}\right\|^2}{\lambda_1\left|\beps^{k+1}\right|^2}\right)\right].
	\end{multline}
	Notice that the last term in \eqref{dec2} is bounded from above by the last term in \eqref{dec3}. Thus, after plugging \eqref{dec2} and \eqref{dec3} into \eqref{dec1}, proceeding as in the proof of inequality \eqref{comppp9} and applying Poincar\'e inequality \eqref{Poincare}, we obtain that
	\be\label{dec4}
	\left(1+\frac{\tau}{4}\left(\beta+\nu\lambda_1\right)\right)\left|\beps^{k+1}\right|^2\leq\left|\beps^{k}\right|^2 \quad \forall k \geq 0.
	\ee
	We conclude the proof by using Lemma \ref{discretegronwall}.
\end{proof}
\begin{rmk}\leavevmode
\begin{enumerate}[(i)]
	\item We notice that it is sufficient to assume a weaker condition on $\tau$, namely, $\tau\beta\leq\Lambda$, where $\Lambda$ is as defined in Proposition \ref{propboundsu}. Nevertheless, we prefer the assumption $\tau\beta\leq1$ for the sake of simplifying the calculations.
	\item We also point out that one can obtain continuous dependence on initial data by using a slightly more general version of Lemma \ref{discretegronwall}. Even though this allows us to eliminate the smallness assumption on the time step, it yields a constant that grows with respect to the number of time steps, as opposed to the decay observed in \eqref{l2decaysemi}.
	\end{enumerate}
\end{rmk}


We now proceed to obtaining error estimates, in the $H$ and $V$ norms, between a solution of \eqref{semimp} and the corresponding continuous in time solution of \eqref{eqDA}. For these proofs, we need to use the uniform bounds of $\rd \bu/\rd t$ and $\rd \vn/\rd t$ from Propositions \ref{propboundsu} and \ref{propboundtimederv}.

\begin{thm}\label{l2convsemi}
Assume hypotheses \eqref{A1}-\eqref{A4} and suppose that $\bu$ satisfies, in addition, bound \eqref{boundderu} for $t\geq0$. Consider $\bv_{N,0} \in P_N H \cap B_V(M_1)$ and let $\vn$ and $\{\bv_N^k\}_{k\in \mathbb{N}}$ be the unique solutions of \eqref{eqDAGalerkin} and \eqref{semimp}, respectively, corresponding to $I_h(\bu)$ and satisfying $\vn(0) = \bv_{N,0} = \bv_N^0$. Let $n_0:= \left \lceil{T_1/\tau}\right \rceil$, with $T_1$ as given in Proposition \ref{propboundtimederv}. Then, for every $n \in \mathbb{N}$ with $n \geq n_0$,
\be\label{semil2errorest}
\left|\bv_N^n - \vn(t_n) \right|^2 \leq \frac{\left|\bv_N^{n_0} - \vn(t_{n_0})\right|^2}{\left(1+\frac{\tau}{4}\left(\beta+\nu\lambda_1\right)\right)^{n-n_0}}+c\tau^2\lambda_1^{-1}R_1^2.
\ee
\end{thm}
\begin{proof}
Let $k\in\mathbb{N}$ be fixed. Integrating equation \eqref{eqDAGalerkin} over $[t_k,t_{k+1}]$ and dividing by $\tau$, we obtain

\begin{multline}\label{comp12a}
  \frac{\bv_N\left(t_{k+1}\right) - \bv_N\left(t_k\right)}{\tau} + \frac{\nu}{\tau} \int_{t_k}^{t_{k+1}} A \bv_N(s) \rd s +\frac{1}{\tau}\int_{t_k}^{t_{k+1}}P_N B\left(\bv_N(s),\bv_N(s)\right) \rd s\\
  = P_N\f - \frac{\beta}{\tau} \int_{t_k}^{t_{k+1}} P_N P_\sigma I_h\left(\bv_N(s)-\bu(s)\right)\rd s.
\end{multline}

We rewrite some of the terms as follows:
\be\label{comp12b}
	A\vn\left(s\right)=A\vn\left(t_{k+1}\right)+A\left(\vn\left(s\right)-\vn\left(t_{k+1}\right)\right),
\ee
\begin{multline}\label{comp12c}
	B\left(\vn\left(s\right),\vn\left(s\right)\right)=B\left(\vn\left(t_k\right),\vn\left(t_{k+1}\right)\right)\\
	+B\left(\vn\left(t_k\right),\vn\left(s\right)-\vn\left(t_{k+1}\right)\right)+B\left(\vn\left(s\right)-\vn\left(t_k\right),\vn\left(s\right)\right),
\end{multline}
\begin{multline}\label{comp12d}
 P_NP_{\sigma}I_h\left(\vn(s)-\bu(s)\right) = P_NP_{\sigma}I_h\left(\vn\left(t_{k+1}\right)-\bu\left(t_{k+1}\right)\right) +\\
 P_NP_{\sigma}I_h\left(\vn(s)-\vn\left(t_{k+1}\right)\right)
  + P_NP_{\sigma}I_h\left(\bu\left(t_{k+1}\right)-\bu(s)\right).
\end{multline}

Hence, we obtain
\begin{multline}\label{comp12}
  \frac{\bv_N\left(t_{k+1}\right) - \bv_N\left(t_k\right)}{\tau} +\nu A \vn\left(t_{k+1}\right)+P_NB\left(\vn\left(t_k\right),\vn\left(t_{k+1}\right)\right)=\\
  P_N\f-\beta P_NP_{\sigma}I_h\left(\vn\left(t_{k+1}\right)-\bu\left(t_{k+1}\right)\right)\\
  -\frac{\beta}{\tau}\int_{t_k}^{t_{k+1}}P_N\left[P_{\sigma}I_h\left(\vn\left(s\right)-\vn\left(t_{k+1}\right)\right)
	+P_{\sigma}I_h\left(\bu\left(t_{k+1}\right)-\bu\left(s\right)\right)\right]\rd s\\
	-\frac{\nu}{\tau} \int_{t_k}^{t_{k+1}} A\left(\vn\left(s\right)-\vn\left(t_{k+1}\right)\right) \rd s \\
	-\frac{1}{\tau}\int_{t_k}^{t_{k+1}}P_N\left[B\left(\vn\left(t_k\right),\vn\left(s\right)-\vn\left(t_{k+1}\right)\right)+B\left(\vn\left(s\right)-\vn\left(t_k\right),\vn\left(s\right)\right)\right]\rd s.
\end{multline}

Subtracting \eqref{comp12} from \eqref{semimp} and writing
\[
 B\left(\vn^{k},\vn^{k+1}\right)-B\left(\vn\left(t_k\right),\vn\left(t_{k+1}\right)\right)=B\left(\vn^k,\bd^{k+1}\right)+B\left(\bd^k,\vn\left(t_{k+1}\right)\right),
\]
we see that the error $\bd^k:= \vn^k - \vn(t_k)$ evolves according to
\begin{multline}\label{comp13}
 \frac{\bd^{k+1} - \bd^k}{\tau} + \nu A \bd^{k+1}=
  \frac{\nu}{\tau} \int_{t_k}^{t_{k+1}} A\left(\bv_N(s) - \bv_N\left(t_{k+1}\right)\right) \rd s\\
 -P_N\left[B\left(\vn^k,\bd^{k+1}\right)+B\left(\bd^k,\vn\left(t_{k+1}\right)\right)\right]+ \\
 + \frac{1}{\tau} \int_{t_k}^{t_{k+1}} P_N \left[B\left(\vn\left(t_k\right),\vn(s)-\vn\left(t_{k+1}\right)\right)+B\left(\vn(s) - \vn\left(t_{k}\right),\vn(s)\right)\right] \rd s \\
 + \frac{\beta}{\tau} \int_{t_k}^{t_{k+1}}  P_N P_\sigma I_h\left(\bv_N(s) - \bv_N\left(t_{k+1}\right)\right) \rd s + \frac{\beta}{\tau} \int_{t_k}^{t_{k+1}}  P_NP_\sigma I_h\left(\bu\left(t_{k+1}\right)-\bu(s)\right) \rd s\\
 -\beta P_N P_{\sigma}I_h\left(\bd^{k+1}\right).
\end{multline}

Taking the inner product of \eqref{comp13} with $2\tau\bd^{k+1}$ for $k\geq n_0$ in $H$, using orthogonality property \eqref{propBorthog2} of the bilinear term and \eqref{displ2}, we obtain
\begin{multline}\label{comp14}
	(1+\tau\beta)\left|\bd^{k+1}\right|^2+\left|\bd^{k+1}-\bd^k\right|^2-\left|\bd^k\right|^2+\tau\nu \left\|\bd^{k+1}\right\|^2\\
	\leq 2 \tau \left| \left(B\left(\bd^k,\vn\left(t_{k+1}\right)\right),\bd^{k+1}\right) \right| + 2 \nu \int_{t_k}^{t_{k+1}} \left| \left( A\left(\bv_N\left(s\right) - \bv_N\left(t_{k+1}\right)\right),\bd^{k+1}\right) \right| \rd s \\
	+ 2\int_{t_k}^{t_{k+1}}\left| \left(B\left(\vn\left(t_k\right),\vn\left(s\right)-\vn\left(t_{k+1}\right)\right),\bd^{k+1}\right)\right| \rd s \\
	 + 2\int_{t_k}^{t_{k+1}} \left| \left( B\left(\vn(s) - \vn\left(t_{k}\right) , \vn(s) \right), \bd^{k+1}\right) \right| \rd s\\
	+ 2 \beta \int_{t_k}^{t_{k+1}} \left| \left(I_h\left(\bv_N\left(s\right) - \bv_N\left(t_{k+1}\right)\right),\bd^{k+1}\right) \right| \rd s \\
	+ 2 \beta \int_{t_k}^{t_{k+1}} \left| \left(I_h\left(\bu\left(t_{k+1}\right) - \bu(s)\right),\bd^{k+1}\right) \right| \rd s.
\end{multline}
We proceed to bound each term in the right-hand side of inequality \eqref{comp14}. First, notice that
\begin{multline}\label{comp15}
	2 \nu \int_{t_k}^{t_{k+1}} \left|\left(A\left(\bv_N(s) - \bv_N\left(t_{k+1}\right)\right),\bd^{k+1}\right)\right| \rd s\\\leq 2 \nu \left\| \bd^{k+1} \right\| \int_{t_k}^{t_{k+1}}\left\| \bv(s)-\vn\left( t_{k+1} \right) \right\| \rd s \\
	\leq 2 \nu \left\| \bd^{k+1} \right\| \tau \int_{t_k}^{t_{k+1}} \left\| \frac{\rd \vn}{\rd s}(s) \right\| \rd s
	\leq c  \nu \left\| \bd^{k+1} \right\| \tau^2 R_1 \leq \frac{\tau\nu}{4} \left\|\bd^{k+1}\right\|^2 + c\tau^3\nu R_1^2,
\end{multline}
where we used the fact that $\bv_N(s)$ is globally Lipschitz in time (with respect to the $V$ norm) with a Lipschitz constant $c_7R_1$ (with $c_7$ being an absolute constant independent on any physical parameter, cf. Proposition \ref{propboundtimederv}) for $t\geq T_1$. As in the proof of \eqref{comppp9}, we assume, without loss of generality, that $\bd^{k+1} \neq 0$ and estimate the third and fourth terms in the right-hand side of \eqref{comp14}, using \eqref{ineqTiti1} and Proposition \ref{propboundtimederv}, by
\begin{multline}\label{comp16}
	 2\int_{t_k}^{t_{k+1}}\left| \left(B\left(\vn\left(t_k\right),\vn\left(s\right)-\vn\left(t_{k+1}\right)\right),\bd^{k+1}\right)\right| \rd s \\
	 + 2\int_{t_k}^{t_{k+1}} \left| \left( B\left(\vn(s) - \vn\left(t_{k}\right) , \vn(s) \right), \bd^{k+1}\right) \right| \rd s \\
	\leq c\int_{t_k}^{t_{k+1}}\left\|\vn\left(t_k\right)\right\|\left\|\vn\left(s\right)-\vn\left(t_{k+1}\right)\right\|\left|\bd^{k+1}\right|\left[ 1 + \log\left( \frac{\left\|\bd^{n+1}\right\|}{\lambda_1^{1/2}\left|\bd^{k+1}\right|}\right) \right]^{1/2}\rd s\\
	\leq\tau\left(\tau R_1\right)\left(cM_1\left|\bd^{k+1}\right|\left[ 1 + \log\left( \frac{\left\|\bd^{k+1}\right\|^2}{\lambda_1\left|\bd^{k+1}\right|^2}\right) \right]^{1/2}\right)\\
	\leq \frac{\tau M_1^2}{\nu}\left|\bd^{k+1}\right|^2\left[ 1 + \log\left( \frac{\left\|\bd^{k+1}\right\|^2}{\lambda_1\left|\bd^{k+1}\right|^2}\right) \right]+c\tau^3\nu R_1^2.
\end{multline}
Using \eqref{leftovl2} with $\alpha_0 = 10$ along with the Poincar\'e inequality \eqref{Poincare}, we obtain
\begin{multline}\label{comp17}
	2\beta\int_{t_k}^{t_{k+1}} \left|\left(I_h\left(\bv_N(s) - \bv_N\left(t_{k+1}\right)\right),\bd^{k+1}\right)\right|\rd s\\
	\leq \int_{t_k}^{t_{k+1}} \left( \frac{\beta}{5}\left|\bd^{k+1}\right|^2 + c \left(\beta+\nu\lambda_1\right)\lambda_1^{-1}\left\|\bv_N\left(s\right) - \bv_N\left(t_{k+1}\right)\right\|^2 \right) \rd s\\
	\leq\frac{\tau\beta}{5}\left|\bd^{k+1}\right|^2+c\tau^3\left(\beta+\nu\lambda_1\right)\lambda_1^{-1}R_1^2.
\end{multline}
Similarly, using the global Lipschitz property in time of $\bu(s)$ from Proposition \ref{propboundsu}, we obtain
\be
  2 \beta \int_{t_k}^{t_{k+1}} \left|\left( I_h\left( \bu\left(t_{k+1}\right) - \bu(s) \right),\bd^{k+1}\right)\right| \rd s \leq \frac{\tau\beta}{5}\left|\bd^{k+1}\right|^2+c\tau^3\left(\beta+\nu\lambda_1\right)\lambda_1^{-1}R_1^2.
\ee
For the first term in the right-hand side of \eqref{comp14}, we write
\begin{multline}\label{comp17a}
	2\tau\left|\left(B\left(\bd^k,\vn\left(t_{k+1}\right)\right),\bd^{k+1}\right)\right|\leq2\tau\left|\left(B\left(\bd^k-\bd^{k+1},\vn\left(t_{k+1}\right)\right),\bd^{k+1}\right)\right|\\
	+2\tau\left|\left(B\left(\bd^{k+1},\vn\left(t_{k+1}\right)\right),\bd^{k+1}\right)\right|.
\end{multline}
Using \eqref{ineqTiti1} and Young's inequality, we obtain
\begin{multline}\label{commonBilinear}
2\tau\left|\left(B\left(\bd^{k+1},\vn\left(t_{k+1}\right)\right),\bd^{k+1}\right)\right|\\
\leq \frac{c\tau M_1^2}{\nu}\left|\bd^{k+1}\right|^2\left[ 1 + \log\left( \frac{\left\|\bd^{k+1}\right\|^2}{\lambda_1\left|\bd^{k+1}\right|^2}\right) \right]+\frac{\tau\nu}{4}\left\|\bd^{k+1}\right\|^2.
\end{multline}
For the other term in the right-hand side of \eqref{comp17a}, we use estimate \eqref{estnonlineartermL2L4L4}, along with the bounds from \eqref{boundsv} and \eqref{semibdsimple}, to obtain
\begin{multline}\label{comp18}
	2 \tau \left|\left(B\left(\bd^k-\bd^{k+1},\vn\left(t_{k+1}\right)\right),\bd^{k+1}\right)\right|\\
	\leq 2\left(\left|\bd^k-\bd^{k+1}\right|\right)\left(c\tau\left\|\bv_N\left(t_{k+1}\right)\right\|^{1/2}\left|A\bv_N\left(t_{k+1}\right)\right|^{1/2}\left\|\bd^{k+1}\right\|^{1/2}\left|\bd^{k+1}\right|^{1/2}\right)\\
	\leq \left|\bd^k-\bd^{k+1}\right|^2+c\tau^2\left\|\bv_N\left(t_{k+1}\right)\right\|\left|A\bv_N\left(t_{k+1}\right)\right|\left\|\bd^{k+1}\right\|\left|\bd^{k+1}\right|\\
	\leq \left|\bd^k-\bd^{k+1}\right|^2+c\tau^2\frac{M_1^3}{\nu^{1/2}}\left(\frac{M_1\Lambda^{1/2}}{\nu^{1/2}}+\beta^{1/2}\right)\left|\bd^{k+1}\right| \\
	\leq\left|\bd^k-\bd^{k+1}\right|^2+\frac{\tau\beta}{10}\left|\bd^{k+1}\right|^2+c\tau^3\frac{M_1^6}{\nu},
\end{multline}
where we used condition \eqref{propvcondbeta} in the last inequality with an appropriately chosen $c$ along with Young's inequality.

Plugging estimates \eqref{comp15}-\eqref{comp18} into \eqref{comp14}, we obtain, after collecting like terms,
\begin{multline*}
	\left(1+\frac{\tau\beta}{2}\right)\left|\bd^{k+1}\right|^2+\left|\bd^k-\bd^{k+1}\right|^2-\left|\bd^k\right|^2+\frac{\tau\nu}{2}\left\|\bd^{k+1}\right\|^2\\
	-\frac{c\tau M_1^2}{\nu}\left|\bd^{k+1}\right|^2\left[ 1 + \log\left(\frac{\left\|\bd^{k+1}\right\|^2}{\lambda_1\left|\bd^{k+1}\right|^2}\right) \right]\\
	\leq\left|\bd^k-\bd^{k+1}\right|^2+c\tau^3(\beta+\nu\lambda_1)\lambda_1^{-1}R_1^2+c\tau^3\frac{M_1^6}{\nu}.
\end{multline*}
Proceeding similarly as in the proof of inequality \eqref{comppp9}, we obtain
\be\label{comp19}
\left(1+\frac{\tau}{4}\left(\beta+\nu\lambda_1\right)\right)\left|\bd^{k+1}\right|^2\leq\left|\bd^k\right|^2+c\tau^3\left(\beta+\nu\lambda_1\right)\lambda_1^{-1}R_1^2+c\tau^3\frac{M_1^6}{\nu}, \quad \forall  k\geq n_0.
\ee
Finally, \eqref{semil2errorest} follows from Lemma \ref{discretegronwall} and by noting that
\[
\frac{M_1^6}{\nu}=c\frac{\nu R_1^2}{\Lambda^2}.
\]
\end{proof}
\begin{thm}\label{h1errorsemi}
	Assume the hypotheses of Theorem \ref{l2convsemi}. Then, we have the following estimate, for every $n \in \mathbb{N}$ with $n\geq n_0$:
	\begin{multline}\label{semih1errorest}
		\left\|\vn^n - \vn(t_n)\right\|^2\leq \frac{\left\|\vn^{n_0} - \vn(t_{n_0})\right\|^2}{\left(1+\frac{\tau}{4}\left(\beta+\nu\lambda_1\right)\right)^{n-n_0}}\\
		+\frac{c\tau M_2^2}{\nu}\left(\Lambda^{-1}+\tau\frac{M_1^2}{\nu}\right)\frac{\left(n-n_0\right)}{\left(1+\frac{\tau}{4}\left(\beta+\nu\lambda_1\right)\right)^{n-n_0}}\left|\vn^{n_0} - \vn(t_{n_0})\right|^2\\
		+ \frac{c\tau^2\nu R_2^2}{\beta+\nu\lambda_1} \left\{ 1+\frac{R_1^2}{\nu R_2^2}\left[\beta+\frac{M_1^2}{\nu}\left(1+\log\left(\frac{M_2}{\lambda_1^{1/2}M_1}\right)\right)\right] \right. \\
	\left. +\frac{M_2^2R_1^2}{\lambda_1\nu^2R_2^2}\left(\tau\frac{M_1^2}{\nu}+\Lambda^{-1}\right)\right\}.
\end{multline}
\end{thm}
\begin{proof}
As in the proof of Theorem \ref{l2convsemi}, we denote $\bd^k = \vn^k - \vn(t_k)$. Then, taking the inner product of \eqref{comp13} with $2\tau A\bd^{k+1}$ in $H$ ($k \geq n_0$) and using \eqref{disph1}, we obtain that
\begin{multline}\label{comp21}
	\left(1+\tau\beta\right)\left\|\bd^{k+1}\right\|^2+\left\|\bd^{k+1}-\bd^k\right\|^2-\left\|\bd^k\right\|^2+\tau \nu \left|A\bd^{k+1}\right|^2\\
	\leq 2\nu\int_{t_k}^{t_{k+1}} \left|\left(A\left(\bv_N\left(s\right) - \bv_N\left(t_{k+1}\right)\right),A\bd^{k+1}\right)\right| \rd s\\
	+2\int_{t_k}^{t_{k+1}} \left|\left(B\left(\vn\left(t_k\right),\vn\left(s\right)-\vn\left(t_{k+1}\right)\right),A\bd^{k+1}\right)\right| \rd s \\
	+2\int_{t_k}^{t_{k+1}} \left|\left(B\left(\vn(s) - \vn\left(t_{k}\right),\vn\left(s\right)\right),A\bd^{k+1}\right)\right| \rd s\\
	+2\beta \int_{t_k}^{t_{k+1}}  \left| \left(I_h\left(\bv_N(s) - \bv_N\left(t_{k+1}\right)\right),A\bd^{k+1}\right) \right| \rd s\\
	+2\beta \int_{t_k}^{t_{k+1}} \left| \left(I_h\left(\bu\left(t_{k+1}\right) - \bu(s)\right),A\bd^{k+1}\right)\right| \rd s\\
	+2\tau\left|\left(B\left(\bv_N^k,\bd^{k+1}\right),A\bd^{k+1}\right)\right|+2\tau\left|\left(B\left(\bd^k,\vn\left(t_{k+1}\right)\right),A\bd^{k+1}\right)\right|.
\end{multline}
Most of the terms on the right-hand side of \eqref{comp21} are estimated similarly as in previous calculations, except we now also use the global Lipschitz property in time of $\bv_N(s)$ with respect to the $A$ norm where appropriate (cf. Proposition \ref{propboundtimederv}). In particular, we have the following estimate of the first term:
\begin{multline}\label{comp22}
	 2\nu\int_{t_k}^{t_{k+1}} \left|\left(A\left(\bv_N\left(s\right) - \bv_N\left(t_{k+1}\right)\right),A\bd^{k+1}\right)\right| \rd s\\
	\leq 2\nu\int_{t_k}^{t_{k+1}} \tau \left| A \frac{\rd \bv_N}{\rd s}(s)\right| \left| A\bd^{k+1}\right| \rd s
	\leq 2c\nu\int_{t_k}^{t_{k+1}}\tau R_2\left|A\bd^{k+1}\right|\rd s\\
	\leq\frac{\tau\nu}{14}\left|A\bd^{k+1}\right|^2+c\tau^3\nu R_2^2,
\end{multline}
Applying \eqref{leftovh1} to both terms involving $I_h$ with $\alpha_0=6$ and $\alpha_1=14$, and using again the global Lipschitz property in time of $\bu(s)$ and $\vn(s)$ with respect to the $V$ norm (Propositions \ref{propboundsu} and \ref{propboundtimederv}) we obtain that
\begin{multline}\label{comp23}
	2\beta \int_{t_k}^{t_{k+1}}  \left| \left(I_h\left(\bv_N(s) - \bv_N\left(t_{k+1}\right)\right),A\bd^{k+1}\right) \right| \rd s\\
	+2\beta \int_{t_k}^{t_{k+1}} \left| \left(I_h\left(\bu\left(t_{k+1}\right) - \bu(s)\right),A\bd^{k+1}\right)\right| \rd s\\
	\leq\frac{\tau\beta}{3}\left\|\bd^{k+1}\right\|^2+\frac{\tau\nu}{7}\left|A\bd^{k+1}\right|^2+c\tau^3\beta R_1^2.
\end{multline}
Using inequalities \eqref{ineqBG} and \eqref{ineqTiti2} to estimate the second and third terms on the right-hand side of \eqref{comp21}, respectively, we have
\begin{multline}\label{comp24}
	2\left|\left(B\left(\vn\left(t_k\right),\vn\left(s\right)-\vn\left(t_{k+1}\right)\right),A\bd^{k+1}\right)\right|\\
	+2\left|\left(B\left(\vn\left(t_k\right)-\vn\left(s\right),\vn\left(s\right)\right),A\bd^{k+1}\right)\right|\\
	\leq 2c_B\left\|\vn\left(t_k\right)\right\|\left\|\vn\left(s\right)-\vn\left(t_{k+1}\right)\right\|\left|A\bd^{k+1}\right|\left[ 1 + \log\left( \frac{\left|A\vn\left(t_k\right)\right|}{\lambda_1^{1/2}\left\|\vn\left(t_k\right)\right\|}\right) \right]^{1/2}\\
	+2c_T\left\|\vn\left(s\right)\right\|\left\|\vn\left(s\right)-\vn\left(t_{k+1}\right)\right\|\left|A\bd^{k+1}\right|\left[ 1 + \log\left( \frac{\left|A\vn\left(s\right)\right|}{\lambda_1^{1/2}\left\|\vn\left(s\right)\right\|}\right) \right]^{1/2}
\end{multline}
Now, we bound $\left|A\bv_N(s)\right|$ and $\left|A\bv_N(t_k)\right|$ by $M_2$ from \eqref{boundsv} and use the fact that the function $\psi(x)=x[1 + \log(\alpha/x)]$ is increasing for $0 \leq x \leq \alpha$, $\alpha>0$, with $\alpha = M_2/\lambda_1^{1/2}$. Hence, since $c M_1 \leq M_2/\lambda_1^{1/2}$, for $0\leq x\leq cM_1 $, $\psi$ attains its maximum at $x=cM_1$. This yields that the right-hand side of \eqref{comp24} is bounded by
\begin{multline}\label{comp24a}
cM_1\left\|\vn\left(s\right)-\vn\left(t_{k+1}\right)\right\|\left|A\bd^{k+1}\right|\left[ 1 + \log\left(\frac{M_2}{\lambda_1^{1/2}M_1}\right) \right]^{1/2}\\
\leq c\tau M_1R_1\left|A\bd^{k+1}\right|\left[ 1 + \log\left(\frac{M_2}{\lambda_1^{1/2}M_1}\right) \right]^{1/2}\\
\leq\frac{\nu}{7}\left|A\bd^{k+1}\right|^2+c\tau^2\frac{M_1^2R_1^2}{\nu}\left[1+\log\left(\frac{M_2}{\lambda_1^{1/2}M_1}\right)\right].
\end{multline}
From \eqref{comp24} and \eqref{comp24a}, we conclude that
\begin{multline}\label{comp25}
	2\int_{t_k}^{t_{k+1}} \left|\left(B\left(\vn\left(t_k\right),\vn\left(s\right)-\vn\left(t_{k+1}\right)\right),A\bd^{k+1}\right)\right| \rd s \\
	+2\int_{t_k}^{t_{k+1}} \left|\left(B\left(\vn(s) - \vn\left(t_{k}\right),\vn\left(s\right)\right),A\bd^{k+1}\right)\right| \rd s\\
	\leq \frac{\tau\nu}{7}\left|A\bd^{k+1}\right|^2+c\tau^3\frac{M_1^2R_1^2}{\nu}\left[1+\log\left(\frac{M_2}{\lambda_1^{1/2}M_1}\right)\right]
\end{multline}
The sixth term on the right-hand side of \eqref{comp21} is bounded by using \eqref{ineqTiti2} and Young's inequality, as
\begin{multline}\label{comp26}
2\tau\left|\left(B\left(\bv_N^k,\bd^{k+1}\right),A\bd^{k+1}\right)\right|\\
\leq\frac{\tau\nu}{14}\left|A\bd^{k+1}\right|^2+\frac{c\tau M_1^2}{\nu}\left\|\bd^{k+1}\right\|^2\left[1+\log\left(\frac{\left|A\bd^{k+1}\right|^2}{\lambda_1\left\|\bd^{k+1}\right\|^2}\right)\right].
\end{multline}
It remains to estimate the last term on the right-hand side of \eqref{comp21}. First, using inequality \eqref{estnonlineartermL4L4L2} and Young's inequality, we obtain that
\begin{multline}\label{comp27}
2\tau\left|\left(B\left(\bd^k,\vn\left(t_{k+1}\right)\right),A\bd^{k+1}\right)\right|\\
\leq 2c\tau\left|\bd^k\right|^{1/2}\left\|\bd^k\right\|^{1/2}\left\|\bv_N\left(t_k\right)\right\|^{1/2}\left|A\bv_N\left(t_k\right)\right|^{1/2}\left|A\bd^{k+1}\right|\\
\leq\frac{\tau\nu}{14}\left|A\bd^{k+1}\right|^2+\frac{c\tau}{\nu}\left|\bd^k\right|\left\|\bd^k\right\|\left\|\bv_N\left(t_k\right)\right\|\left|A\bv_N\left(t_k\right)\right|.
\end{multline}
Now, using the bounds from \eqref{boundsv} and writing
\[ \left\|\bd^k\right\|\leq \left\|\bd^{k+1}\right\|+\left\|\bd^{k+1}-\bd^k\right\|,
\]
we see that the second term on the right hand side of \eqref{comp27} is bounded by
\begin{multline*}
\frac{c\tau}{\nu}\left|\bd^k\right|\left\|\bd^k-\bd^{k+1}\right\|M_1M_2+\frac{c\tau}{\nu}\left|\bd^k\right|\left\|\bd^{k+1}\right\|M_1M_2\\
\leq \left\|\bd^k-\bd^{k+1}\right\|^2+\frac{c\tau^2}{\nu^2}\left|\bd^k\right|^2M_1^2M_2^2+\tau\frac{cM_1^2}{\nu}\Lambda\left\|\bd^{k+1}\right\|^2+\frac{c\tau M_2^2}{\nu\Lambda}\left|\bd^k\right|^2\\
\leq \left\|\bd^k-\bd^{k+1}\right\|^2+\frac{\tau\beta}{6}\left\|\bd^{k+1}\right\|^2+\frac{c\tau M_2^2}{\nu}\left(\tau\frac{M_1^2}{\nu}+\Lambda^{-1}\right)\left|\bd^k\right|^2,
\end{multline*}	
where we used condition \eqref{propvcondbeta} on $\beta$. Therefore, we have
\begin{multline}\label{comp28}
	2\tau\left|\left(B\left(\bd^k,\vn\left(t_{k+1}\right)\right),A\bd^{k+1}\right)\right|\leq \frac{\tau\nu}{14}\left|A\bd^{k+1}\right|^2+\left\|\bd^k-\bd^{k+1}\right\|^2\\
	+\frac{\tau\beta}{6}\left\|\bd^{k+1}\right\|^2+\frac{c\tau M_2^2}{\nu}\left(\tau\frac{M_1^2}{\nu}+\Lambda^{-1}\right)\left|\bd^k\right|^2.
\end{multline}

Plugging \eqref{comp22}, \eqref{comp23}, \eqref{comp25}, \eqref{comp26} and \eqref{comp28} into \eqref{comp21}, we obtain
\begin{multline*}
	\left(1+\frac{\tau\beta}{2}\right)\left\|\bd^{k+1}\right\|^2+\frac{\tau\nu}{2}\left|A\bd^{k+1}\right|^2-\frac{c\tau M_1^2}{\nu}\left\|\bd^{k+1}\right\|^2\left[1+\log\left(\frac{\left|A\bd^{k+1}\right|^2}{\lambda_1\left\|\bd^{k+1}\right\|^2}\right)\right]\\
	\leq\left\|\bd^k\right\|^2+c\tau^3\nu R_2^2+c\tau^3\beta R_1^2+\frac{c\tau^3M_1^2R_1^2}{\nu}\left[1+\log\left( \frac{M_2}{\lambda_1^{1/2}M_1}\right) \right]\\
	+\frac{c\tau M_2^2}{\nu}\left(\tau\frac{M_1^2}{\nu}+\Lambda^{-1}\right)\left|\bd^k\right|^2.
\end{multline*}
Proceeding as in the proof of inequality \eqref{comppp9} and using Poincar\'e inequality, \eqref{Poincare}, we obtain that
\begin{multline}\label{comp28a}
	\left(1+\frac{\tau}{4}\left(\beta+\nu\lambda_1\right)\right)\left\|\bd^{k+1}\right\|^2\leq\left\|\bd^k\right\|^2 + c\tau^3\nu R_2^2 \\
	+ c\tau^3R_1^2\left[\beta+\frac{M_1^2}{\nu}\left(1+\log\left(\frac{M_2}{\lambda_1^{1/2}M_1}\right)\right)\right]
	+\frac{c\tau M_2^2}{\nu}\left(\tau\frac{M_1^2}{\nu}+\Lambda^{-1}\right)\left|\bd^k\right|^2, \quad \forall k \geq n_0.
\end{multline}
From \eqref{comp28a} and Lemma \eqref{discretegronwall}, it follows that, for every $n \geq n_0$,
\begin{multline}\label{comp29}
	\left\|\bd^{n}\right\|^2 \leq \frac{\left\|\bd^{n_0}\right\|^2}{\left(1+\gamma\right)^{n-n_0}}+\frac{c\tau M_2^2}{\nu}\left(\tau\frac{M_1^2}{\nu}+\Lambda^{-1}\right)\sum_{k=n_0}^{n-1}\frac{\left|\bd^k\right|^2}{\left(1+\gamma\right)^{n-k}}+\frac{c\tau^2\nu R_2^2}{\beta+\nu\lambda_1}\\
	+\frac{c\tau^2R_1^2}{\beta+\nu\lambda_1}\left[\beta+\frac{M_1^2}{\nu}\left(1+\log\left(\frac{M_2}{\lambda_1^{1/2}M_1}\right)\right)\right],
\end{multline}
where $\gamma=\tau(\beta+\nu\lambda_1)/4$. In order to estimate the summation appearing in \eqref{comp29}, we use the result from Theorem \ref{l2convsemi} and obtain that
\begin{multline}\label{comp30}
\sum_{k=n_0}^{n-1}\frac{|\bd^k|^2}{\left(1+\gamma\right)^{n-k}}\leq \sum_{k=n_0}^{n-1}\frac{|\bd^{n_0}|^2}{\left(1+\gamma\right)^{k-n_0+n-k}} + \sum_{k=n_0}^{n-1}\frac{c\tau^2\lambda_1^{-1}R_1^2}{(1+\gamma)^{n-k}}\\
\leq \frac{n-n_0}{\left(1+\gamma\right)^{n-n_0}}\left|\bd^{n_0}\right|^2+\frac{c\tau\lambda_1^{-1}R_1^2}{\beta+\nu\lambda_1}.
\end{multline}
Finally, \eqref{semih1errorest} follows by plugging \eqref{comp30} into \eqref{comp29}.
\end{proof}

Next, we consider a fully discrete approximation, i.e. in space and time, of \eqref{eqDA} by using the time discretization scheme \eqref{semimp} and the spatial discretization given by the Postprocessing Galerkin method (subsection \ref{subsecPPGM}). Combining the results from Theorems \ref{l2convsemi} and \ref{h1errorsemi} with the error estimates for the Postprocessing Galerkin method from Theorem \ref{thmerrorPPGM}, we are able to show error estimates, again in the $H$ and $V$ norms, between this fully discrete approximation of a solution $\vn$ of \eqref{eqDA} and the corresponding reference solution $\bu$ of \eqref{eqNSE}.

\begin{thm}\label{thmerrorfulldiscSI}
	Assuming the hypotheses of Theorems \ref{thmerrorPPGM} and \ref{l2convsemi}, there exists $T_3 = T_3(\nu,\lambda_1,|\f|,N, \tau) \geq 0$ such that, for every $n \geq \lceil T_3/ \tau \rceil $,
	\be\label{errorfulldiscSIH}
		|\vn^n + \Phi_1(\vn^n) - \bu(t_n)| \leq c\tau\lambda_1^{-1/2} R_1 +  C \frac{L_N}{\lambda_{N+1}^{5/4}}
    \ee
    and
    \begin{multline}\label{errorfulldiscSIV}
    	\|\vn^n + \Phi_1(\vn^n) - \bu(t_n)\| \leq \\
    	\leq \frac{c\tau R_2 \nu^{1/2} }{(\beta+\nu\lambda_1)^{1/2}} \left\{ 1+\frac{R_1^2}{\nu R_2^2}\left[\beta+\frac{M_1^2}{\nu}\left(1+\log\left(\frac{M_2}{\lambda_1^{1/2}M_1}\right)\right)\right] \right. \\
    	\left. +\frac{M_2^2R_1^2}{\lambda_1\nu^2R_2^2}\left(\tau\frac{M_1^2}{\nu}+\Lambda^{-1}\right)\right\}^{1/2} + C \frac{L_N}{\lambda_{N+1}^{3/4}},
    \end{multline}
    where $c$ is an absolute constant independent on any physical parameter and $C$ is a constant depending on $\nu$, $\lambda_1$, $|\f|$ and $1/h^2$, but independent of $N$.
\end{thm}
\begin{proof}
Let $\vn$ be the unique solution of \eqref{eqDAGalerkin} corresponding to $I_h(\bu)$ and satisfying $\vn(0) = \bv_{N,0}$. Notice that
\begin{multline}\label{errorFDSIH}
	|\vn^n + \Phi_1(\vn^n) - \bu(t_n)| \leq \\
	\leq |\vn^n - \vn(t_n)| + |\Phi_1(\vn^n) - \Phi_1(\vn(t_n))| + |\vn(t_n) + \Phi_1(\vn(t_n)) - \bu(t_n)|\\
	\leq (1 + l)|\vn^n - \vn(t_n)| + |\vn(t_n) + \Phi_1(\vn(t_n)) - \bu(t_n)|,
\end{multline}
where $l > 0$ is the Lipschitz constant of $\Phi_1$ as given in \eqref{Phi1LipsH} and \eqref{Phi1LipsV}. Hence, \eqref{errorfulldiscSIH} follows from \eqref{errorFDSIH} and the results of Theorems \ref{thmerrorPPGM} and \ref{l2convsemi}. Clearly, \eqref{errorfulldiscSIV} follows analogously, but using the result of Theorem \ref{h1errorsemi} instead.
\end{proof}

\subsection{Fully implicit in time scheme}\label{subsecFI}

Let us again consider a time step $\tau > 0$ and a regular sequence of times $t_k = k \tau$, for every $k \in \mathbb{N}$. The fully implicit in time Euler scheme is given by
\be\label{fullimp}
  \frac{\bv_N^{k+1} - \bv_N^k}{\tau} + \nu A \bv_N^{k+1} + P_N B(\bv_N^{k+1},\bv_N^{k+1}) = P_N \f - \beta P_N P_\sigma I_h(\bv_N^{k+1} - \bu(t_{k+1})).
\ee

Notice that the difference with respect to the semi-implicit scheme \eqref{semimp} lies in the discretization of the bilinear term, since now both entries evolve at the same time.

Next, we show existence of a solution of the initial-value problem associated to \eqref{fullimp}. First, we state the following lemma, whose proof can be found, e.g., in \cite[Lemma 7.2]{ConstantinFoiasbook}.

\begin{lem}\label{exis}
	Let $\Xi$ be a finite-dimensional inner product space, with inner product $(\,\cdot \, ,\,  \cdot\,)_\Xi$. Let $B \subset \Xi$ be a closed ball. Suppose $\Phi:B\rightarrow \Xi$ is continuous and satisfies $(\Phi (\bxi),\bxi)_\Xi <0$ for every $\bxi \in \partial B$. Then, there exists $\bxi \in B$ such that $\Phi(\bxi)=0$.
\end{lem}

\begin{prop}\label{exsfull}
Let $\bu$ be a solution of \eqref{eqNSE} on $[0, \infty)$ and assume hypotheses \eqref{A2} and (\textbf{A4}). Suppose that $\beta > 0$ and $h > 0$ satisfy $c_0 \beta h^2 \leq \nu$. Then, given $\bv_{N,0} \in P_N H$, there exists a sequence $\{\bv_N^k\}_{k \in \mathbb{N}}$ that solves \eqref{fullimp} and satisfies $\bv_N^0 = \bv_{N,0}$.	
\end{prop}
\begin{proof}
By induction, it suffices to prove that, given $\bv_N^k \in P_N H$, there exists $\bv_N^{k+1} \in P_N H$ satisfying \eqref{fullimp}.
	
Let $R = \tau|\f|+|\bv_N^k|+\tau \beta |I_h\left(\bu\left(t_{k+1}\right)\right)|+\nu$, and denote by $B_{P_N H}(R)$ the ball of radius $R$ centered at $0$ in $P_N H$. Define $\Phi : B_{P_N H}(R) \rightarrow P_N H$ by
\[
\Phi(\bxi):=P_N\f+\frac{\bv_N^k-\bxi}{\tau}-\nu A\bxi-P_NB\left(\bxi,\bxi\right)-\beta P_N P_\sigma I_h\left(\bxi-\bu\left(t_{k+1}\right)\right).
\]
Taking the inner product of $\Phi(\bxi)$ with $\bxi$ in $H$, for $\bxi \in B_{P_N H}(R)$, and using \eqref{displ2}, we obtain that
\begin{multline}
\left(\Phi\left(\bxi\right),\bxi\right)
\leq\left(\f,\bxi\right)+\frac{1}{\tau}\left(\bv_N^k,\bxi\right)-\frac{1}{\tau}\left|\bxi\right|^2-\frac{\nu}{2} \left\|\bxi\right\|^2-\frac{\beta}{2}|\bxi|^2 + \beta (I_h(\bu(t_{k+1})),\bxi) \\
\leq \frac{1}{\tau}\left(\tau\left|\f\right|+\left|\bv_N^k\right|+\tau\beta\left|I_h\left(\bu(t_{k+1})\right)\right|-\left|\bxi\right|\right)\left|\bxi\right|-\frac{\nu}{2}\left\|\bxi\right\|^2-\frac{\beta}{2}\left|\bxi\right|^2.
\end{multline}
Hence, for $|\bxi|=R$, we have
\[
(\Phi(\bxi),\bxi)\leq - \frac{\nu R}{\tau} < 0.
\]
Therefore, by Lemma \ref{exis}, there exists $\bv_N^{k+1} \in P_N H$ satisfying \eqref{fullimp}.
\end{proof}

Next, we would like to show uniqueness of solutions to the nonlinear initial-value problem that arises when \eqref{fullimp} is complemented with some initial data, say $\bv_{N,0}$. First, we will show that any sequence $\{\vn^k\}_{k \in \mathbb{N}}$ that solves \eqref{fullimp} and satisfies $\bv_N^0=\bv_{N,0}$ is uniformly bounded, with respect to $k$, $N$ and $\tau$, in the $V$ norm. For this end, we denote by $\{\widetilde{\bv}_N^k\}_{k \in \mathbb{N}}$ the unique solution of the semi-implicit scheme \eqref{semimp} that was established in Proposition \ref{uniqsemi} satisfying the same initial data as the solution $\{\vn^k\}_{k \in \mathbb{N}}$ of \eqref{fullimp}, i.e. $\widetilde{\bv}_N^0 = \vn^0=\bv_{N,0}$. The idea consists in showing that $|\widetilde\bv_N^k-\bv_N^k|$ and $\tau\|\widetilde\bv_N^k-\bv_N^k\|$ satisfy \eqref{semifulll2} and \eqref{semifullh1} below, respectively. Then, using this fact along with the uniform boundedness of $\{\widetilde{\bv}_N^k\}_{k \in \mathbb{N}}$ in the $V$ norm from Theorem \ref{semibounds}, we will show, via an inductive argument, that $\{\vn^k\}_{k \in \mathbb{N}}$ is uniformly bounded,  in $V$ as well. A crucial factor in obtaining this result is the fact that the uniform bounds of $\{\widetilde{\bv}_N^k\}_{k \in \mathbb{N}}$ from Theorem \ref{semibounds} are independent of the nudging parameter $\beta$, which in turn is possible due to the stabilizing mechanism imposed by the feedback-control (nudging) term.

We start by proving a preliminary inequality that allows the use of an inductive argument.

\begin{prop}\label{premh1bd}
Assume hypotheses \eqref{A1}-\eqref{A4} and let $\bv_{N,0} \in P_N H \cap B_V(M_1)$ be given. Let $\{\bv_N^k\}_{k\in \mathbb{N}}$ be any solution of \eqref{fullimp} corresponding to $I_h(\bu)$ and satisfying $\bv_N^0 = \bv_{N,0}$ (observe that such a solution exists by Proposition \ref{exsfull}). Then,
\be\label{fullprelh1}
 \left\|\bv_N^{k+1}\right\|^2\leq 4\left\|\bv_N^{k}\right\|^2+40M_1^2, \quad \forall k \in \mathbb{N}.
\ee
\end{prop}
\begin{proof}
Let $\{\widetilde{\bv}_N^k\}_{k \in \mathbb{N}}$ be the unique solution of \eqref{semimp} corresponding to $I_h(\bu)$ and satisfying $\widetilde{\bv}_N^0 = \bv_{N,0}$. Set $\bta^k := \widetilde\bv_N^k-\bv_N^k$, $k \in \mathbb{N}$. Subtracting \eqref{fullimp} from \eqref{semimp}, we see that $\{\bta^k\}_{k\in \mathbb{N}}$ satisfies	
\begin{multline}\label{semifull}
	\frac{\bta^{k+1} - \bta^k}{\tau} + \nu A \bta^{k+1} + P_N\left[ B\left(\widetilde\bv_N^k-\widetilde\bv_N^{k+1},\widetilde\bv_N^{k+1}\right)+B\left(\widetilde\bv_N^{k+1},\bta^{k+1}\right) \right.\\
	+B\left(\bta^{k+1},\widetilde\bv_N^{k+1}\right)
	-\left.B\left(\bta^{k+1},\bta^{k+1}\right)\right]= -\beta P_N P_\sigma I_h\left(\bta^{k+1}\right).
\end{multline}
Taking the inner product of \eqref{semifull} with $2\tau\bta^{k+1}$ in $H$, using orthogonality property \eqref{propBorthog2}, inequality \eqref{displ2}, and neglecting $\left|\bta^{k+1}-\bta^k\right|^2$ from the left-hand side, we obtain that
\begin{multline*}
	\left(1+\tau\beta\right)\left|\bta^{k+1}\right|^2+\tau\nu\left\|\bta^{k+1}\right\|^2\leq \left|\bta^k\right|^2\\
	+2\tau\left|\left(B\left(\bta^{k+1},\widetilde\bv_N^{k+1}\right),\bta^{k+1}\right)\right|+2\tau\left|\left(B\left(\widetilde\bv_N^k-\widetilde\bv_N^{k+1},\widetilde\bv_N^{k+1}\right),\bta^{k+1}\right)\right|.
\end{multline*}
Now, using inequalities \eqref{ineqTiti1} and \eqref{semibdsimple} for estimating the bilinear terms $B$, we have
\begin{multline*}
	\left(1+\tau\beta\right)\left|\bta^{k+1}\right|^2+\tau\nu\left\|\bta^{k+1}\right\|^2\leq \left|\bta^k\right|^2+\frac{\tau\nu}{2}\left\|\bta^{k+1}\right\|^2+\frac{\tau\nu}{4}M_1^2\\
	+\frac{c\tau M_1^2}{\nu}\left|\bta^{k+1}\right|^2\left[1+\log\left(\frac{\left\|\bta^{k+1}\right\|^2}{\lambda_1\left|\bta^{k+1}\right|^2}\right)\right].
\end{multline*}
Proceeding as in the proof of inequality \eqref{comppp9}, we obtain
\be\label{comp31}
	\left(1+\frac{\tau\beta}{4}\right)\left|\bta^{k+1}\right|^2+\frac{\tau\nu}{4}\left\|\bta^{k+1}\right\|^2\leq \left|\bta^k\right|^2+\frac{\tau\nu}{4}M_1^2, \quad \forall k \in\mathbb{N}.
\ee
Applying the Poincar\'e inequality \eqref{Poincare} to the second term on the left-hand side of \eqref{comp31} and using Lemma \ref{discretegronwall}, yields
\be\label{semifulll2}
 \left|\bta^{k+1}\right|^2\leq \frac{\nu M_1^2}{\beta+\nu\lambda_1}, \quad \forall k \in \mathbb{N},
\ee
where we used that $\bta^0=0$. Plugging estimate \eqref{semifulll2} into \eqref{comp31}, it follows in particular that
\be\label{semifullh1}
 \tau\left\|\bta^{k+1}\right\|^2\leq \frac{4M_1^2}{\beta+\nu\lambda_1}+\tau M_1^2, \quad \forall k\in\mathbb{N}.
\ee
Next, taking the inner product of \eqref{fullimp} with $2\tau A\bv_N^{k+1}$ in $H$ and proceeding similarly as in \eqref{compp9}, we obtain
\begin{multline}\label{comp32}
	\left(1+\frac{\tau\beta}{2}\right)\left\|\bv_N^{k+1}\right\|^2+\frac{3\tau\nu}{4}\left|A\bv_N^{k+1}\right|^2 \\
	\leq \left\|\bv_N^{k}\right\|^2+\frac{8\tau}{\nu}\left|\f\right|^2+10\tau\beta M_1^2
+2\tau\left|\left(B\left(\bv_N^{k+1},\bv_N^{k+1}\right),A\bv_N^{k+1}\right)\right|\\
\leq \left\|\bv_N^{k}\right\|^2+8\tau\nu\lambda_1M_1^2+10\tau\beta M_1^2\\
+2\tau\left|\left(B\left(\widetilde\bv_N^{k+1},\bv_N^{k+1}\right),A\bv_N^{k+1}\right)\right|+2\tau\left|\left(B\left(\bta^{k+1},\bv_N^{k+1}\right),A\bv_N^{k+1}\right)\right|,
\end{multline}
where in the last inequality we used that $\vn^{k+1} = \widetilde{\bv}_N^{k+1} - \bta^{k+1}$ and $\left|\f\right|\leq \nu\lambda_1^{1/2}M_1$.

Using inequality \eqref{ineqTiti2}, along with the uniform bound \eqref{semibdsimple}, we obtain
\begin{multline}\label{comp33}
	2\tau\left|\left(B\left(\widetilde\bv_N^{k+1},\bv_N^{k+1}\right),A\bv_N^{k+1}\right)\right|\leq \frac{\tau\nu}{8}\left|A\bv_N^{k+1}\right|^2 \\
	+\frac{c\tau}{\nu}M_1^2\left\|\bv_N^{k+1}\right\|^2\left(1+\log\left(\frac{\left|A\bv_N^{k+1}\right|^2}{\lambda_1\left\|\bv_N^{k+1}\right\|^2}\right)\right).
\end{multline}
Now, using \eqref{estnonlineartermL4L4L2}, \eqref{semifulll2} and \eqref{semifullh1}, we have
\begin{multline}\label{comp33a}
	2\tau\left|\left(B\left(\bta^{k+1},\bv_N^{k+1}\right),A\bv_N^{k+1}\right)\right|\leq c\tau\left|\bta^{k+1}\right|^{1/2}\left\|\bta^{k+1}\right\|^{1/2}\left\|\bv_N^{k+1}\right\|^{1/2}\left|A\bv_N^{k+1}\right|^{3/2}\\
	\leq \frac{\tau\nu}{8}\left|A\bv_N^{k+1}\right|^2+\frac{c\tau}{\nu^3}\left|\bta^{k+1}\right|^2\left\|\bta^{k+1}\right\|^2\left\|\bv_N^{k+1}\right\|^2\\
	\leq \frac{\tau\nu}{8}\left|A\bv_N^{k+1}\right|^2+\frac{c}{\nu^3}\left(\frac{\nu M_1^2}{\beta+\nu\lambda_1}\right)\left(\tau\left\|\bta^{k+1}\right\|^2\right)\left\|\bv_N^{k+1}\right\|^2\\
	\leq \frac{\tau\nu}{8}\left|A\bv_N^{k+1}\right|^2+\frac{cM_1^2}{\nu^2\left(\beta+\nu\lambda_1\right)}\left(\frac{4M_1^2}{\beta+\nu\lambda_1}+\tau M_1^2\right)\left\|\bv_N^{k+1}\right\|^2\\
	\leq \frac{\tau\nu}{8}\left|A\bv_N^{k+1}\right|^2+\frac{cM_1^4}{\nu^2\left(\beta+\nu\lambda_1\right)^2}\left\|\bv_N^{k+1}\right\|^2+\frac{c\tau M_1^4}{\nu^2\left(\beta+\nu\lambda_1\right)}\left\|\bv_N^{k+1}\right\|^2.
\end{multline}
Moreover, using condition \eqref{propvcondbeta} on $\beta$ with a suitable absolute constant $c$, it follows from \eqref{comp33a} that
\be\label{comp34}
2\tau\left|\left(B\left(\bta^{k+1},\bv_N^{k+1}\right),A\bv_N^{k+1}\right)\right|\leq\frac{\tau\nu}{8}\left|A\bv_N^{k+1}\right|^2+\frac{3}{4}\left\|\bv_N^{k+1}\right\|^2+\frac{\tau M_1^2}{\nu}\left\|\bv_N^{k+1}\right\|^2.
\ee
Noting that the last term on the right-hand side of inequality \eqref{comp34} is dominated by the last term of inequality \eqref{comp33}, we obtain after plugging \eqref{comp34} and \eqref{comp33} into \eqref{comp32},
\begin{multline*}
	\left(1+\frac{\tau\beta}{2}\right)\left\|\bv_N^{k+1}\right\|^2+\frac{3\tau\nu}{4}\left|A\bv_N^{k+1}\right|^2\leq \left\|\bv_N^{k}\right\|^2+8\tau\nu\lambda_1M_1^2+10\tau\beta M_1^2\\
+\frac{\tau\nu}{4}\left|A\bv_N^{k+1}\right|^2+\frac{3}{4}\left\|\bv_N^{k+1}\right\|^2+\frac{c\tau}{\nu}M_1^2\left\|\bv_N^{k+1}\right\|^2\left(1+\log\left(\frac{\left|A\bv_N^{k+1}\right|^2}{\lambda_1\left\|\bv_N^{k+1}\right\|^2}\right)\right).
\end{multline*}
Adding similar terms, proceeding as in the proof of inequality \eqref{comppp9}, and using Poincar\'e inequality \eqref{Poincare} we obtain
\[
	\frac{1}{4}\left(1+\tau\left(\beta+\nu\lambda_1\right)\right)\left\|\bv_N^{k+1}\right\|^2\leq \left\|\bv_N^{k}\right\|^2+8\tau\nu\lambda_1M_1^2+10\tau\beta M_1^2,
\]
from which \eqref{fullprelh1} follows immediately.
\end{proof}

In the next theorem, we prove that any solution, $\{\vn^k\}_{k \in \mathbb{N}}$, of \eqref{fullimp} is bounded uniformly in $H$ and $V$, for all $k\in\mathbb{N}$, $\tau>0$ and $N\in\mathbb{Z^+}$. In particular, Proposition \ref{premh1bd} plays a crucial role in obtaining the uniform bound in the $V$ norm.

\begin{thm}\label{fullbounds}
Assume the hypotheses of Proposition \ref{premh1bd}. Then, for every $n \in \mathbb{N}$,
\be\label{fulll2bd}
\left|\bv_N^n\right|^2\leq \frac{\left|\bv_0\right|^2}{\left(1+\frac{\tau}{2}\left(\beta+2\nu \lambda_1\right)\right)^n}+\frac{12\left|\f\right|^2}{\beta\left(\beta+2\nu\lambda_1\right)}+\frac{12\beta M_0^2}{\beta+2\nu\lambda_1}+\frac{12\nu M_1^2}{\beta+2\nu\lambda_1}
\ee
and
\be\label{fullh1bd}
\left\|\bv_N^n\right\|^2\leq \frac{\left\|\bv_0\right\|^2}{\left(1+\frac{\tau}{4}\left(\beta+\nu\lambda_1\right)\right)^n}+\frac{24\left|\f\right|^2}{\nu\left(\beta+\nu\lambda_1\right)}+\frac{32\beta M_1^2}{\beta+\nu\lambda_1}.
\ee
In particular,
\be\label{fullbdsimple}
\left|\bv_N^n\right|\leq \lambda_1^{-1/2}\left\|\bv_N^n\right\|\leq 6\lambda_1^{-1/2}M_1 \quad \forall n \in \mathbb{N}.
\ee
\end{thm}
\begin{proof}
As a consequence of the orthogonality property \eqref{propBorthog2}, the proof of inequality \eqref{fulll2bd} is exactly the same as that of inequality \eqref{semil2bd}; thus it will be omitted.

We prove inequality \eqref{fullh1bd} by an inductive argument similar to the one used in the proof of \eqref{semih1bd}. Notice that inequality \eqref{fullh1bd} is trivially true for $n=0$. Now, fix $n\in\mathbb{N}$ and suppose \eqref{fullh1bd} is true for $k\in\{0,\ldots,n\}$. Taking the inner product of equation \eqref{fullimp} with $2\tau A\bv_N^{k+1}$ in $H$, we obtain, similarly as in \eqref{compp9},
\begin{multline}\label{comp35}
	\left(1+\frac{\tau\beta}{2}\right)\left\|\bv_N^{k+1}\right\|^2+\frac{2\tau\nu}{3}\left|A\bv_N^{k+1}\right|^2\leq \left\|\bv_N^{k}\right\|^2+\frac{6\tau}{\nu}\left|\f\right|^2+8\tau\beta M_1^2\\
+2\tau\left|\left(B\left(\bv_N^{k+1},\bv_N^{k+1}\right),A\bv_N^{k+1}\right)\right|.
\end{multline}
In order to estimate the last term in the right-hand side of \eqref{comp35}, we use \eqref{ineqTiti2} and the preliminary inequality \eqref{fullprelh1} as follows:
\begin{multline}\label{comp36}
	2\tau\left|\left(B\left(\bv_N^{k+1},\bv_N^{k+1}\right),A\bv_N^{k+1}\right)\right|\leq c\tau\left\|\bv_N^{k+1}\right\|^2\left|A\bv_N^{k+1}\right|\left[1+\log\left(\frac{\left|A\bv_N^{k+1}\right|^2}{\lambda_1\left\|\bv_N^{k+1}\right\|^2}\right)\right]^{1/2}\\
	\leq\frac{\tau\nu}{6}\left|A\bv_N^{k+1}\right|^2+\frac{c\tau}{\nu}\left\|\bv_N^{k+1}\right\|^4\left(1+\log\left(\frac{\left|A\bv_N^{k+1}\right|^2}{\lambda_1\left\|\bv_N^{k+1}\right\|^2}\right)\right)\\
	\leq \frac{\tau\nu}{6}\left|A\bv_N^{k+1}\right|^2+\frac{c\tau}{\nu}\left(4\left\|\bv_N^k\right\|^2+40M_1^2\right)\left\|\bv_N^{k+1}\right\|^2\left(1+\log\left(\frac{\left|A\bv_N^{k+1}\right|^2}{\lambda_1\left\|\bv_N^{k+1}\right\|^2}\right)\right).
\end{multline}
Since $\left\|\bv_0\right\|\leq M_1$, it follows from the induction hypothesis, along with \eqref{boundf}, that $\left\|\bv_N^{k}\right\|\leq 6M_1$ for all $k\in \{0,\ldots,n\}$. Using this in \eqref{comp36} and plugging the resulting estimate in \eqref{comp35}, we obtain
\begin{multline*}
	\left(1+\frac{\tau\beta}{2}\right)\left\|\bv_N^{k+1}\right\|^2+\frac{\tau\nu}{2}\left|A\bv_N^{k+1}\right|^2\leq \left\|\bv_N^{k}\right\|^2+\frac{6\tau}{\nu}\left|\f\right|^2+8\tau\beta M_1^2\\
+\frac{c\tau M_1^2}{\nu}\left\|\bv_N^{k+1}\right\|^2\left(1+\log\left(\frac{\left|A\bv_N^{k+1}\right|^2}{\lambda_1\left\|\bv_N^{k+1}\right\|^2}\right)\right).
\end{multline*}
We now proceed exactly as in the proof of inequality \eqref{semih1bd} to close the inductive argument.
\end{proof}

Now, using the result of Theorem \ref{fullbounds}, we are able to prove continuous dependence on the initial data of solutions of \eqref{fullimp}. In particular, this implies uniqueness of a solution of the initial-value problem associated to \eqref{fullimp}.

\begin{thm}\label{fulldecay}
Assume hypotheses \eqref{A1}-\eqref{A4}. Consider $\bv_{N,0}, \overline{\bv}_{N,0} \in P_N H \cap B_V(M_1)$ and let $\{\bv_N^k\}_{k\in \mathbb{N}}$ and $\{\overline{\bv}_N^k\}_{k \in \mathbb{N}}$ be any two solutions of \eqref{fullimp} corresponding to $I_h(\bu)$ and with initial conditions $\bv_{N,0}$ and $\overline{\bv}_{N,0}$, respectively. Then,
\be\label{fulldecayineq}
\left\|\overline{\bv}_N^n - \bv_N^n\right\|^2 \leq \frac{\left\| \overline{\bv}_{N,0} - \bv_{N,0} \right\|^2}{\left(1+\frac{\tau}{4}\left(\beta+\nu\lambda_1\right)\right)^{n}}, \quad \forall n \in \mathbb{N}.
\ee
\end{thm}
\begin{proof}
Denote $\beps^k:= \overline{\bv}_N^k - \bv_N^k$. Notice that $\{\beps^k\}_{k \in \mathbb{N}}$ satisfies
\begin{multline}\label{difffull}
\frac{\beps^{k+1}-\beps^k}{\tau}	+\nu A\beps^{k+1}+B\left(\bv_N^{k+1},\beps^{k+1}\right)+B\left(\beps^{k+1},\bv_N^{k+1}\right)+B\left(\beps^{k+1},\beps^{k+1}\right)\\
=-\beta P_{\sigma}I_h\left(\beps^{k+1}\right) \quad \forall k \geq 0.
\end{multline}
Taking the inner product of \eqref{difffull} with $2\tau A\beps^{k+1}$ in $H$, we obtain, after using \eqref{displ2},
\begin{multline}\label{difffull2}
	\left(1+\tau\beta\right)\left\|\beps^{k+1}\right\|^2+\tau\nu\left|A\beps^{k+1}\right|^2\leq \left\|\beps^{k}\right\|^2+2\tau\left|\left(B\left(\bv_N^{k+1},\beps^{k+1}\right),A\beps^{k+1}\right)\right|\\
	+2\tau\left|\left(B\left(\beps^{k+1},\bv_N^{k+1}\right),A\beps^{k+1}\right)\right|+2\tau\left|\left(B\left(\beps^{k+1},\beps^{k+1}\right),A\beps^{k+1}\right)\right|.
\end{multline}
Now, using inequality \eqref{ineqTiti2} to estimate the second term on the right-hand side of \eqref{difffull2} and \eqref{ineqBG} to estimate the third and fourth terms, along with the uniform bound of $\{\| \bv_N^k \|\}_{k \in \mathbb{N}}$ from Theorem \ref{fullbounds}, we obtain
\begin{multline*}
	\left(1+\tau\beta\right)\left\|\beps^{k+1}\right\|^2+\tau\nu\left|A\beps^{k+1}\right|^2\leq \left\|\beps^{k}\right\|^2+\frac{\tau\nu}{2}\left|A\beps^{k+1}\right|^2\\
	+\frac{c\tau}{\nu}M_1^2\left\|\beps^{k+1}\right\|^2\left(1+\log\left(\frac{\left|A\beps^{k+1}\right|^2}{\lambda_1\left\|\beps^{k+1}\right\|^2}\right)\right),
\end{multline*}	
Then, proceeding as in the proof of inequality \eqref{comppp9}, we obtain
\begin{equation}\label{difffull3}
	\left(1+\frac{\tau}{4}\left(\beta+\nu\lambda_1\right)\right)\left\|\beps^{k+1}\right\|^2\leq \left\|\beps^{k}\right\|^2, \quad \forall k \in\mathbb{N}.
\end{equation}
Finally, \eqref{fulldecayineq} follows from \eqref{difffull3} and Lemma \ref{discretegronwall}.
\end{proof}

In the next theorem, we estimate the error between a solution of \eqref{fullimp} and the corresponding continuous in time solution of \eqref{eqDA}, with respect to the norms in $H$ and $V$.

\begin{thm}\label{thmL2convFullImp}
Assume hypotheses \eqref{A1}-\eqref{A4} and suppose that $\bu$ satisfies, in addition, bound \eqref{boundderu} for $t\geq0$. Consider $\bv_{N,0} \in P_N H \cap B_V(M_1)$ and let $\vn$ and $\{\bv_N^k\}_{k\in \mathbb{N}}$ be the unique solutions of \eqref{eqDAGalerkin} and \eqref{fullimp}, respectively, corresponding to $I_h(\bu)$ and satisfying $\vn(0) = \bv_{N,0} = \bv_N^0$. Let $n_0:= \left \lceil{T_1/\tau}\right \rceil$, with $T_1$ as given in Proposition \ref{propboundtimederv}. Then, for every $n \in \mathbb{N}$ with $n \geq n_0$,
\be\label{thmL2convFIres}
\left| \vn^n - \vn(t_n) \right|^2 \leq \frac{\left| \vn^{n_0} -\vn(t_{n_0}) \right|^2}{\left(1+\frac{\tau}{4}\left(\beta+\nu\lambda_1\right)\right)^{n-n_0}}+c\tau^2\lambda_1^{-1}R_1^2,
\ee
and
\begin{multline}\label{thmH1convFIres}
\left\|\vn^n - \vn(t_n) \right\|^2\leq \frac{\left\| \vn^{n_0} -\vn(t_{n_0}) \right\|^2}{\left(1+\frac{\tau}{4}\left(\beta+\nu\lambda_1\right)\right)^{n-n_0}}\\
+\frac{c\tau^2\nu R_2^2}{\beta+\nu\lambda_1}\left[1+\frac{\beta R_1^2}{\nu R_2^2}+\frac{M_1^2R_1^2}{\nu^2R_2^2}\left(1+\log\left(\frac{M_2}{M_1\lambda_1^{1/2}}\right)\right)\right].
\end{multline}
\end{thm}
\begin{proof}
Denote $\bd^k = \vn^k - \vn(t_k)$, $k \in \mathbb{N}$. Notice that
\begin{multline}\label{comp37a}
  B\left(\vn^{k+1},\vn^{k+1}\right) - B\left(\vn(s),\vn(s)\right) = \\
  \left[ B\left(\vn^{k+1},\vn^{k+1}\right)-B\left(\vn\left(t_{k+1}\right),\vn\left(t_{k+1}\right)\right) \right] + \\
  + \left[ B\left(\vn\left(t_{k+1}\right),\vn\left(t_{k+1}\right)\right) - B\left(\vn(s),\vn(s)\right)\right] \\
  = \left[ B\left(\bd^{k+1},\bd^{k+1}\right)+B\left(\vn\left(t_{k+1}\right),\bd^{k+1}\right)+B\left(\bd^{k+1},\vn\left(t_{k+1}\right)\right) \right] \\
  - \left[ B\left(\vn\left(t_{k+1}\right),\vn\left(s\right)-\vn\left(t_{k+1}\right)\right)+B\left(\vn\left(s\right)-\vn\left(t_{k+1}\right),\vn\left(s\right)\right), \right].
\end{multline}

Then, proceeding as in \eqref{comp12a}-\eqref{comp12} and using \eqref{comp37a}, we see that $\{\bd^k\}_{k \in \mathbb{N}}$ evolves according to
\begin{multline}\label{comp37}
 \frac{\bd^{k+1} - \bd^k}{\tau} + \nu A \bd^{k+1}=
  \frac{\nu}{\tau} \int_{t_k}^{t_{k+1}} A(\bv_N(s) - \bv_N(t_{k+1})) \rd s\\
 -P_N\left[B\left(\bd^{k+1},\bd^{k+1}\right)+B\left(\vn\left(t_{k+1}\right),\bd^{k+1}\right)+B\left(\bd^{k+1},\vn\left(t_{k+1}\right)\right)\right]+ \\
 + \frac{1}{\tau} \int_{t_k}^{t_{k+1}} P_N B\left(\vn\left(t_{k+1}\right),\vn(s)-\vn\left(t_{k+1}\right) \right) \rd s\\
 + \frac{1}{\tau} \int_{t_k}^{t_{k+1}} P_N B\left(\vn(s) - \vn\left(t_{k+1}\right) , \vn(s) \right) \rd s \\
 + \frac{\beta}{\tau} \int_{t_k}^{t_{k+1}}  P_N P_\sigma I_h\left(\bv_N\left(s\right) - \bv_N\left(t_{k+1}\right)\right) \rd s \\
 + \frac{\beta}{\tau} \int_{t_k}^{t_{k+1}}  P_NP_\sigma I_h\left(\bu\left(t_{k+1}\right)-\bu\left(s\right)\right) \rd s-\beta P_NP_{\rho}I_h\left(\bd^{k+1}\right),
\end{multline}
for every $k\in \mathbb{N}$. Taking the inner product of \eqref{comp37} with $2 \tau \bd^{k+1}$ in $H$ for $k\geq n_0$, we see that all the terms can be handled in exactly the same way as in the proof of Theorem \ref{l2convsemi}, except for the ones involving the bilinear terms from the second line of \eqref{comp37}. From these, the first two ones vanish by virtue of orthogonality property \eqref{propBorthog2}, and the third one is estimated as in \eqref{commonBilinear}. We omit the details.

In order to prove error estimate \eqref{thmH1convFIres}, we first take the inner product of \eqref{comp37} with $2\tau A\bd^{k+1}$ in $H$. Again, we notice that all the terms can be handled as in the proof of Theorem \ref{h1errorsemi}, except for the ones involving the bilinear terms from the second line of \eqref{comp37}. Therefore, proceeding analogously, we obtain
\begin{multline}\label{comp37b}
	\left(1+\frac{\tau\beta}{2}\right)\left\|\bd^{k+1}\right\|^2+\frac{11 \tau\nu}{16}\left|A\bd^{k+1}\right|^2-\left\|\bd^k\right\|^2\leq 2\tau\left|\left(B\left(\bd^{k+1},\bd^{k+1}\right),A\bd^{k+1}\right)\right|\\
	+2\tau\left|\left(B\left(\bd^{k+1},\bv_N\left(t_{k+1}\right)\right),A\bd^{k+1}\right)\right|+2\tau\left|\left(B\left(\bv_N\left(t_{k+1}\right),\bd^{k+1}\right),A\bd^{k+1}\right)\right|\\
	+c\tau^3\frac{M_1^2R_1^2}{\nu}\left[1+\log\left(\frac{M_2}{\lambda_1^{1/2}M_1}\right)\right]+c\tau^3\nu R_2^2+c\tau^3\beta R_1^2.
\end{multline}
By Proposition \ref{propboundtimederv} and Theorem \ref{fullbounds}, we have
\be\label{comp37c}
  \left\|\bd^{k}\right\|\leq c M_1, \quad \forall k \in \mathbb{N}.
\ee

Using \eqref{ineqBG} and \eqref{comp37c}, we estimate the first term on the right-hand side of \eqref{comp37b} as
\begin{multline*}
2\tau\left|\left(B\left(\bd^{k+1},\bd^{k+1}\right),A\bd^{k+1}\right)\right|+2\tau\left|\left(B\left(\bd^{k+1},\bv_N\left(t_{k+1}\right)\right),A\bd^{k+1}\right)\right|\leq \\
c\tau M_1\left\|\bd^{k+1}\right\|\left|A\bd^{k+1}\right|\left[1+\log\left(\frac{\left|A\bd^{k+1}\right|^2}{\lambda_1\left\|\bd^{k+1}\right\|^2}\right)\right]^{1/2}\\
\leq \frac{\tau\nu}{8}\left|A\bd^{k+1}\right|^2+\frac{c\tau M_1^2}{\nu}\left\|\bd^{k+1}\right\|^2\left[1+\log\left(\frac{\left|A\bd^{k+1}\right|^2}{\lambda_1\left\|\bd^{n+1}\right\|^2}\right)\right].
\end{multline*}
The second and third terms on the right-hand side of \eqref{comp37b} are estimated similarly, but using \eqref{ineqTiti2} instead. Now, proceeding as in the proof of inequality \eqref{comppp9}, using Poincar\'e inequality \eqref{Poincare} and some algebraic manipulations, we obtain
\begin{multline}\label{comp37d}
	\left(1+\frac{\tau}{4}\left(\beta+\nu\lambda_1\right)\right)\left\|\bd^{k+1}\right\|^2\leq \left\|\bd^{k}\right\|^2 \\
	+c\tau^3\nu R_2^2\left[1+\frac{\beta R_1^2}{\nu R_2^2}+\frac{M_1^2R_1^2}{\nu^2R_2^2}\left(1+\log\left(\frac{M_2}{M_1\lambda_1^{1/2}}\right)\right)\right] \quad \forall k \geq n_0.
\end{multline}
Finally, \eqref{thmH1convFIres} follows from \eqref{comp37d} and Lemma \ref{discretegronwall}.
\end{proof}

Finally, we consider a fully discrete approximation of a solution $\vn$ of \eqref{eqDA} by using the fully implicit in time Euler scheme \eqref{fullimp} and the Postprocessing Galerkin method. Then, combining the results from Theorems \ref{thmerrorPPGM} and \ref{thmL2convFullImp}, we obtain estimates of the error, in the $H$ and $V$ norms, between this fully discrete (in space and time) approximation of a solution $\bv$ of \eqref{eqDA} and the corresponding reference solution $\bu$ of \eqref{eqNSE}.

\begin{thm}\label{thmerrorfulldiscFI}
	Assuming the hypothesis of Theorems \ref{thmerrorPPGM} and \ref{thmL2convFullImp}, there exists $T_4 = T_4(\nu,\lambda_1,|\f|,N, \tau) \geq 0$ such that, for every $n \geq \lceil T_4/ \tau \rceil $,
	\be\label{errorfulldiscFIH}
	|\vn^n + \Phi_1(\vn^n) - \bu(t_n)| \leq c\tau\lambda_1^{-1/2} R_1 +  C \frac{L_N}{\lambda_{N+1}^{5/4}}
	\ee
	and
	\begin{multline}\label{errorfulldiscFIV}
		\|\vn^n + \Phi_1(\vn^n) - \bu(t_n)\| \leq \\
		\leq \frac{c\tau \nu^{1/2} R_2}{(\beta+\nu\lambda_1)^{1/2}}\left[1+\frac{\beta R_1^2}{\nu R_2^2}+\frac{M_1^2R_1^2}{\nu^2R_2^2}\left(1+\log\left(\frac{M_2}{M_1\lambda_1^{1/2}}\right)\right)\right]^{1/2} + C \frac{L_N}{\lambda_{N+1}^{3/4}},
	\end{multline}
	where $C$ is a constant depending on $\nu$, $\lambda_1$, $|\f|$ and $1/h^2$, but independent of $N$.
\end{thm}

\setcounter{equation}{0}
\appendix

\section*{Appendix}
\renewcommand{\thesection}{A}
\renewcommand{\theequation}{{A.}\arabic{equation}}

We now present a proof of Proposition \ref{propboundtimederv}. We start with some related terminology. For every vector space $X$, we denote its complexification by $X_\mathbb{C}$, i.e.
\[
 X_{\mathbb{C}} = \{ \bu + i \bv \,: \, \bu \in X\,, \,\, \bv \in X \}.
\]
Similarly, if $\mathcal{T}:X \rightarrow Y$ is a linear map between vector spaces $X$ and $Y$, we denote by $\mathcal{T}_\mathbb{C} : X_{\mathbb{C}} \to Y_{\mathbb{C}}$ its complexification, given by
\[
 \mathcal{T}_{\mathbb{C}} (\bu + i \bv ) = \mathcal{T}(\bu) + i \mathcal{T}(\bv) \quad \forall \bu, \bv \in X.
\]

Consider $\bu_0 \in H$ and let $\bu$ be the solution of \eqref{eqNSE} on $(0, \infty)$ satisfying $\bu(0) = \bu_0$. It was proven in \cite{foiastemam1979} (see also \cite[Chapter 12]{ConstantinFoiasbook}, \cite{FoiasJollyLanRupamYangZhang2014} and \cite{FoiasManleyTemam1988}) that there exists a neighborhood $\mB$ of $(0,\infty)$ in $\mathbb{C}$ and a unique extension of $\bu$ to $\mB \cup \{0\}$ given by the unique solution of
\be\label{compeqNSE}
\frac{\rd \widetilde\bu}{\rd \xi}(\xi) + \nu A_\mathbb{C} \widetilde\bu(\xi) + B_\mathbb{C}\left(\widetilde\bu(\xi),\widetilde\bu(\xi)\right) = \f, \quad \xi \in \mathcal{B},
\ee
\be\label{compeqNSEic}
\widetilde{\bu}(0) = \bu_0.
\ee
Moreover, $\widetilde{\bu}$ is an analytic $\mD(A)_{\mathbb{C}}$-valued function on $\mB$.


The next proposition provides uniform bounds of $\widetilde{\bu}$ and $\rd \widetilde{\bu}/\rd \xi$ with respect to the norm in $V_{\mathbb{C}}$, which are valid on suitable subsets of $\mB \cup \{0\}$. The proof is given in \cite[Appendix]{FoiasManleyTemam1988}. From now on, for simplicity, we abuse notation and drop the subindex ``$\mathbb{C}$'' from the complexified form of the functional spaces and operators.

First, let us consider $T_{0,1} = T_{0,1}(\nu, \lambda_1, G, |\bu_0|) \geq 0$ such that (see Proposition \ref{propboundsu})
\be\label{bounduV}
 \|\bu(t)\| \leq M_1 \quad \forall t \geq T_{0,1}.
\ee

\begin{prop}\label{apppropboundstildeu}
 Let $\bu_0 \in H$ and let $\widetilde{\bu}$ be the unique solution of \eqref{compeqNSE} on $\mB$ satisfying $\widetilde{\bu}(0) = \bu_0$. Then,
 \be\label{boundtildeuV}
   \|\widetilde{\bu}(\xi)\| \leq 2 M_1 \quad \forall \xi \in \mB_1^0 \subset \mB \cup \{0\},
 \ee
 where
 \be
   \mB_1^0 = \left\{ \xi = t_0 + s \Exp^{i\theta} \in \mathbb{C} \,:\, t_0 \geq T_{0,1}, \theta \in [-\pi/4,\pi/4], s \in [0,\rho]\right\},
 \ee
 with $T_{0,1} \geq 0$ being the same from \eqref{bounduV} and $\rho$ defined by
 \be\label{defrhoradiusanal}
  \rho:= \cos \theta \left\{ c\nu \lambda_1 \left( G + \frac{M_1^2}{\nu^2 \lambda_1} \right) \left[ 1 + \log\left( G + \frac{M_1^2}{\nu^2 \lambda_1} \right) \right]\right\}^{-1} .
 \ee
 Moreover, for every subset $K \subset \mB_1^0$ such that $r = \dist(K,\partial \mB_1^0) > 0$, we have
 \be\label{bounddertildeuV}
   \left\| \frac{\rd \widetilde{\bu}}{\rd \xi}(\xi) \right\| \leq c \frac{M_1}{r} \quad \forall \xi \in K.
 \ee
\end{prop}

\begin{rmk}
Choosing, for example, $K = [ T_{0,1} + \rho/\sqrt{2}, \infty)$ and noticing that $r = \dist(K,\partial \mB_1^0) = \rho/2$, we see that the uniform bound \eqref{boundderu} from Proposition \ref{propboundsu} follows from \eqref{bounddertildeuV} provided $T_0 \geq T_{0,1} + \rho/\sqrt{2}$.
\end{rmk}

In order to prove Proposition \ref{propboundtimederv}, we consider a solution  $\bu$ of \eqref{eqNSE} on $[0, \infty)$ satisfying
\be\label{conduboundV}
 \|\bu(t)\| \leq M_1 \quad \forall t \geq 0.
\ee
In this case, similarly as in Proposition \ref{apppropboundstildeu}, one can show that the unique extension of $\bu$ to $\mB \cup \{0\} \subset \mathbb{C}$ satisfies
\be\label{boundtildeuV2}
\|\widetilde{\bu}(\xi)\| \leq 2 M_1 \quad \forall \xi \in \mB_1 \subset \mB \cup \{0\},
\ee
where
\be\label{deftildeB1}
  \mB_1 = \left\{ \xi = t_0 + s \Exp^{i\theta} \in \mathbb{C} \,:\, t_0 \geq 0, \theta \in [-\pi/4,\pi/4], s \in [0,\rho]\right\}.
\ee

Notice now that equation \eqref{eqDAGalerkin} is a finite dimensional ODE. Classical ODE theory tells us that the complexified version of the equation (when complemented with an initial value) has a unique solution that is analytic in some neighborhood $\widetilde{\mB}$ of $(0,\infty)$ such that  $\widetilde{\mB} \subset \mB$ (since the analyticity of $\widetilde{\bu}$ now controls the analyticity of the right-hand side). That is to say, given $\bv_0 \in H$ and $\vn$ the unique solution of \eqref{eqDAGalerkin} corresponding to $I_h (\bu)$ and satisfying $\bv_N(0) = P_N \bv_0$, there exists a neighborhood $\widetilde{\mB}$ of $(0,\infty)$, with $\widetilde{\mB} \subset \mB$, and a unique analytic extension of $\vn$ to $\widetilde{\mB} \cup \{0\}$, $\widetilde\bv_N:\widetilde{\mB} \cup \{0\} \rightarrow\mD(A)_{\mathbb{C}}$, that satisfies
\begin{multline}\label{compeqDAGalerkin1}
	\frac{\rd \widetilde{\bv_N}}{\rd \xi} (\xi)+ \nu A\widetilde{\bv_N}(\xi) + P_N B\left(\widetilde{\bv_N}(\xi),\widetilde{\bv_N}(\xi)\right) \\
	= P_N \f - \beta P_N P_{\sigma} I_{h}\left(\widetilde{\bv_N}(\xi) - \widetilde{\bu}(\xi)\right), \quad \xi \in \widetilde{\mB},
\end{multline}
\be\label{compeqDAGalerkin2}
\widetilde{\bv_N}(0)= P_N \bv_0.
\ee
We prove in Proposition \ref{apppropboundstildevN} below that the set $\widetilde{\mB}$ does not depend on $N\in\mathbb{Z^+}$ by obtaining uniform bounds of the solution and its derivative in various norms on some subsets of $\widetilde{\mB} \cup \{0\}$ that do not depend on $N$. We remark that the proof of \eqref{boundtildevNDA}, below, follows by a slight modification of the argument used in \cite[Lemma 4.4]{FoiasJollyLanRupamYangZhang2014}.
\begin{prop}\label{apppropboundstildevN}
Assume hypotheses \eqref{A1}-\eqref{A3}, and let $\widetilde{\bu}$ be the unique solution of \eqref{compeqNSE} satisfying $\widetilde{\bu}(0) = \bu(0)$. Consider $\bv_0 \in B_V(M_1)$ and, given $N \in \mathbb{Z^+}$, let $\widetilde{\bv_N}$ be the unique solution of \eqref{compeqDAGalerkin1}-\eqref{compeqDAGalerkin2}. Then,
 \be\label{boundtildevNV}
   \|\widetilde{\bv_N}(\xi)\| \leq 13 M_1 \quad \forall \xi \in \mB_1,
 \ee
with $\mB_1$ as defined in \eqref{deftildeB1}, and
 \be\label{boundtildevNDA}
  |A \widetilde{\bv_N}(\xi)| \leq c M_2  \quad \forall \xi \in \mB_2,
 \ee
where
 \be
   \mB_2 = \left\{ \xi \in \mathbb{C} \,:\, |\Re(\xi)|\geq \frac{\rho}{\sqrt{2}}, \,\,  |\Im(\xi)| \leq \frac{\rho}{2 \sqrt{2}} \right\} \subset \mB_1.
 \ee
Moreover, for every subsets $K_1 \subset \mB_1$ and $K_2 \subset \mB_2$, with $r_1 = \dist(K_1,\partial \mB_1) > 0$ and $r_2 = \dist(K_2, \partial \mB_2) > 0$, we have
 \be\label{bounddertildevNV}
   \left\| \frac{\rd \widetilde{\bv_N}}{\rd \xi}(\xi) \right\| \leq c \frac{M_1}{r_1} \quad \forall \xi \in K_1
 \ee
and
 \be\label{bounddertildevNDA}
   \left|A \frac{\rd \widetilde{\bv_N}}{\rd \xi}(\xi) \right| \leq c \frac{M_2}{r_2} \quad \forall \xi \in K_2,
 \ee
 where $M_1$ and $M_2$ are as in Proposition \ref{propboundtimederv}.
\end{prop}
\begin{proof}
Given $t_0 > 0$ and $|\theta| \leq \pi/4$, let $\widetilde{\rho} = \widetilde{\rho}(t_0, \theta) > 0$ be such that $t_0 + s \Exp^{i \theta} \in \widetilde{\mB}$ for every $s \in (0, \widetilde{\rho})$. We start by taking the inner product of \eqref{compeqDAGalerkin1} with $A \widetilde{\bv_N}(\xi)$, for $\xi = t_0 + s \Exp^{i\theta}$ with $t_0 > 0$, $|\theta| \leq \pi/4$ and $s \in (0, \widetilde{\rho})$. Then, multiplying by $\Exp^{i \theta}$ and taking the real part, it follows that
\begin{multline}\label{apppropboundstildevNineqdervN}
 \frac{1}{2} \frac{\rd}{\rd s}\|\widetilde{\bv_N}\|^2 + \nu \cos \theta |A\widetilde{\bv_N}|^2 + \beta \cos \theta \|\widetilde{\bv_N}\|^2 \leq \\
 \leq |(\f, A\tvN)| + |(B(\tvN,\tvN), A\tvN)| + \beta |(I_h(\tvN) - \tvN, A \tvN)| + \beta |(I_h(\widetilde{\bu}), A \tvN)|.
\end{multline}

Applying Cauchy-Schwarz, Young's inequality and, in particular, inequality \eqref{ineqBG} to estimate the second term on the right-hand side of \eqref{apppropboundstildevNineqdervN}, as well as property \eqref{propIh} of $I_h$ and hypothesis \eqref{propvcondbetah} to estimate the last two terms, we obtain that
\begin{multline}\label{apppropboundstildevNineqdervN2}
  \frac{\rd}{\rd s}\|\widetilde{\bv_N}\|^2 + \frac{\nu}{5} |A\tvN|^2 + \frac{\beta}{5} \|\tvN\|^2 \\
  \leq 15 \beta \|\widetilde{\bu}\|^2 + 15 \frac{|\f|^2}{\nu} + c \frac{\|\tvN\|^4}{\nu} \left[ 1 + \log \left( \frac{|A\tvN|}{\lambda_1^{1/2} \|\tvN\|} \right)\right],
\end{multline}
where we have also used that $\cos \theta \geq 1/\sqrt{2}$ and \eqref{boundtildeuV2}.

Since $s \in (0, \widetilde{\rho}) \mapsto \|\tvN(t_0 + s \Exp^{i \theta})\|$ is continuous and, by \eqref{boundsv}, $\tvN(t_0) \in B_V(8 M_1)$, there exists $s' \in (0, \widetilde{\rho})$ such that
\[
  \|\tvN(t_0 + s \Exp^{i \theta})\| \leq 14 M_1 \quad \forall s \in [0, s'].
\]
Thus, we can define
\be
  s^{\ast} = \sup\{ s' \in (0, \widetilde{\rho}) \,: \, \|\tvN(t_0  + s \Exp^{i \theta})\| \leq 14 M_1 \,\, \forall s \in [0, s'] \}.
\ee
Suppose that $s^{\ast} < \rho$, with $\rho$ as given in \eqref{defrhoradiusanal}. Hence, from \eqref{boundtildeuV2} and \eqref{apppropboundstildevNineqdervN2}, we obtain that, for all $s \in [0, s^{\ast}]$,
\begin{multline}\label{comp38}
   \frac{\rd}{\rd s}\|\widetilde{\bv_N}\|^2 + \frac{\nu}{10} |A\tvN|^2 + \frac{\beta}{5} \|\tvN\|^2 \\
   + \frac{\nu \lambda_1}{10} \left\{ \frac{|A\tvN|^2}{\lambda_1 \|\tvN\|^2} - c \frac{M_1^2}{\nu^2 \lambda_1} \left[1 + \log \left( \frac{|A\tvN|^2}{\lambda_1 \|\tvN\|^2} \right) \right] \right\} \|\tvN\|^2 \leq 15 \beta M_1^2 + 15 \frac{|\f|^2}{\nu}.
\end{multline}
Then, using \eqref{ineqminlog} and hypothesis \eqref{propvcondbeta}, yields
\be\label{comp38a}
  \frac{\rd}{\rd s}\|\widetilde{\bv_N}\|^2 + \frac{\nu}{10} |A\tvN|^2 + \frac{\beta}{10} \|\tvN\|^2 \leq  15 (\beta +\nu \lambda_1) M_1^2,
\ee
where we have also used \eqref{boundf} in order to estimate the last term in the right-hand side of \eqref{comp38}. Then, applying Poincar\'e inequality to the second term on the left-hand side of \eqref{comp38a}, we obtain
\be\label{apppropboundstildevNineqdervN3}
  \frac{\rd}{\rd s}\|\tvN\|^2 + \frac{\beta + \nu \lambda_1}{10} \|\tvN\|^2 \leq 15 (\beta +\nu \lambda_1) M_1^2.
\ee
This implies that
\be\label{apppropboundstildevNineqdervN5}
  \|\tvN (t_0 + s \Exp^{i \theta})\|^2 \leq \|\tvN (t_0)\|^2 \Exp^{-\frac{\beta + \nu \lambda_1}{10}s} + 150 M_1^2  (1 - \Exp^{- \frac{\beta + \nu \lambda_1}{10}s}) \quad \forall s \in [0, s^{\ast}].
\ee
Thus, in particular,
\be
  \|\tvN(t_0 + s^{\ast} \Exp^{i\theta})\| \leq 13 M_1,
\ee
which, by the definition of $s^{\ast}$, is a contradiction. Therefore, $s^{\ast} \geq \rho$ and
\be\label{apppropboundstildevNineqdervN6}
  \|\tvN(\xi)\| \leq 13 M_1 \quad \forall \xi \in \mB_1.
\ee
As a consequence, $\mB_1 \subset \widetilde{\mB} \cup \{0\}$. Now, \eqref{bounddertildevNV} follows from \eqref{boundtildevNV} and Cauchy's integral formula.

Next, we show inequality \eqref{boundtildevNDA}. First, let
\be
\mB_1^{\ast} = \mB_1 \cap  \left\{ \xi \in \mathbb{C} \,:\, |\Im (\xi)| \leq \rho/(2 \sqrt{2}) \right\}.
\ee
Considering \eqref{comp38a} for $\xi = \xi_0 + s\Exp^{i \frac{\pi}{4}} \in \mB_1$, with $\xi_0 \in \mB_1^{\ast}$ and $s \in [0, \rho/4]$, and integrating with respect to $s$ on $[0, \rho/4]$, we obtain in particular that
\be\label{apppropboundstildevNineqvN}
 \frac{\nu}{10} \int_0^{\rho/4} |A\tvN(\xi_0 + s\Exp^{i \frac{\pi}{4}})|^2 \rd s \leq \|\tvN(\xi_0)\|^2 + 15 \frac{\rho}{4} (\beta + \nu \lambda_1) M_1^2.
\ee
Hence, using \eqref{apppropboundstildevNineqdervN6}, it follows that
\be\label{apppropestintAtildevN}
  \int_0^{\rho/4} |A\tvN(\xi_0 + s\Exp^{i \frac{\pi}{4}})|^2 \rd s \leq c [1 + \rho(\beta + \nu \lambda_1)] \frac{M_1^2}{\nu} \quad \forall \xi_0  \in \mB_1^{\ast}.
\ee
Analogously, we can show that
\be\label{apppropestintAtildeu2}
\int_0^{\rho/4} |A\widetilde{\vn}(\xi_0 + s \Exp^{- i\frac{\pi}{4}})|^2 \rd s \leq c [1 + \rho(\beta + \nu \lambda_1)] \frac{M_1^2}{\nu} \quad \forall \xi_0  \in \mB_1^{\ast}.
\ee

Now, consider $\zeta = t_0 + i b_0 \in \mB_2$ and let $D(\zeta, \rho/(4 \sqrt{2})) \subset \mB_1$ be the disc centered at $\zeta$ with radius $\rho/(4 \sqrt{2})$. Let $abcdef$ be the polygon shown in the picture below.

\begin{figure}[h!]
	\centering
	\includegraphics[width=0.6\textwidth]{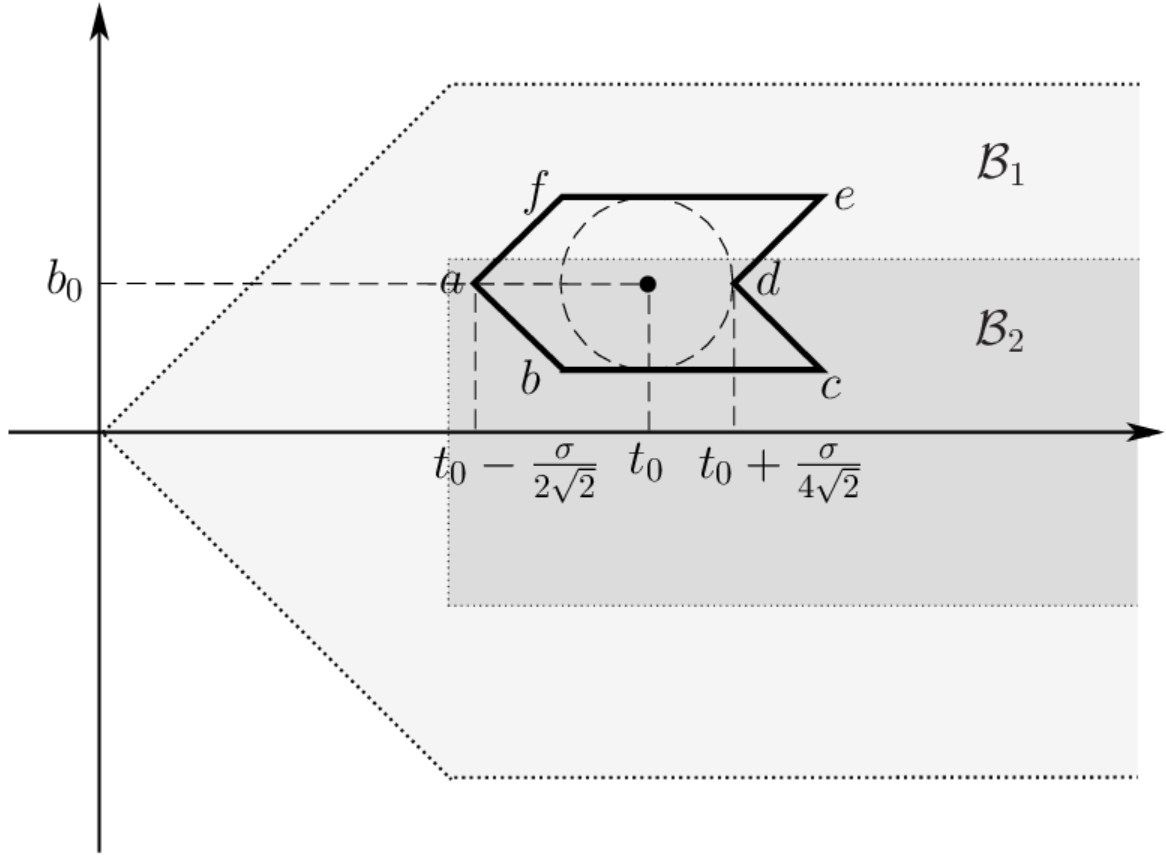}
\end{figure}

Since $A\widetilde{\vn}$ is analytic on $D(\zeta, \rho/(4 \sqrt{2}))$, it follows by the mean value property that
\begin{multline}
	|A\widetilde{\vn}(\zeta)| \leq \frac{32}{\pi \rho^2} \int \int_{D(\zeta, \rho/(4 \sqrt{2}))} |A\widetilde{\vn}(\xi)|\rd \xi \\
	\leq \frac{32}{\pi \rho^2} \left[ \int \int_{\textnormal{adef}} |A\widetilde{\vn}(\xi)| \rd \xi + \int \int_{\textnormal{abcd}} |A\widetilde{\vn}(\xi)| \rd \xi \right].
\end{multline}

Let $\xi \in adef$. We introduce the change of variables:
\be
\xi = t + i b_0 + s \Exp^{i \frac{\pi}{4}}, \quad t \in \left[ t_0 - \frac{\rho}{2 \sqrt{2}}, t_0 + \frac{\rho}{4 \sqrt{2}} \right], \,\, s \in [0, \rho/4].
\ee
Now, using \eqref{apppropestintAtildevN}, we obtain that
\begin{multline}
	\frac{32}{\pi \rho^2} \int \int_{\textnormal{adef}} |A\widetilde{\vn}(\xi)| \rd \xi
	=  \frac{32}{\sqrt{2} \pi \rho^2} \int_{t_0 - \rho/(2 \sqrt{2})}^{t_0 + \rho/(4 \sqrt{2})} \int_0^{\rho/4} |A\widetilde{\vn}(t + i b_0 + s \Exp^{i\frac{\pi}{4}})| \rd s \rd t \\
	\leq c \frac{M_1}{\nu^{1/2} \rho^{1/2}} + c (\beta + \nu \lambda_1)^{1/2} \frac{M_1}{\nu^{1/2}}\\
	 \leq  c\frac{M_1^2}{\nu} \left[ 1 + \log \left( \frac{M_1}{\nu \lambda1^{1/2}} \right) \right]^{1/2} + c \beta^{1/2} \frac{M_1}{\nu^{1/2}},
\end{multline}
where in the last inequality we used the definition of $\rho$ from \eqref{defrhoradiusanal}. Analogously, using \eqref{apppropestintAtildeu2}, one can show that
\be
\frac{32}{\pi \rho^2} \int \int_{\textnormal{abcd}} |A\widetilde{\vn}(\xi)| \rd \xi \leq c\frac{M_1^2}{\nu} \left[ 1 + \log \left( \frac{M_1}{\nu \lambda_1^{1/2}} \right) \right]^{1/2} + c \beta^{1/2} \frac{M_1}{\nu^{1/2}}.
\ee

Therefore, we conclude that
\be
|A\widetilde{\vn}(\zeta)| \leq M_2 \quad \forall \zeta \in \mB_2,
\ee
as desired.

Finally, \eqref{bounddertildevNDA} follows from \eqref{boundtildevNDA} and a direct application of Cauchy's integral formula.
\end{proof}

\begin{rmk}
	We notice that, after a suitable limiting process, a result analogous to Proposition \ref{apppropboundstildevN} is valid for the solution $\bv$ of \eqref{eqDA}. However, since here we are only interested in the Galerkin approximation $\vn$ of $\bv$, we avoid dealing with such technical details.	
\end{rmk}

\begin{rmk}
	Using arguments similar to the ones from the proof of Proposition \ref{apppropboundstildevN}, one can also show that, up to possibly different absolute constants, the same upper bounds from \eqref{boundtildevNDA} and \eqref{bounddertildevNDA} hold for $|A \widetilde{\bu}(\cdot)|$ and $|A\rd \widetilde{\bu}/\rd \xi (\cdot)|$, respectively, with $\widetilde{\bu}$ being a solution of \eqref{compeqNSE}. In particular, this yields uniform bounds of $|A \bu(\cdot)|$ and $|A\rd \bu/\rd \xi (\cdot)|$, with $\bu$ being a solution of \eqref{eqNSE}, which are sharper than the bounds derived in \cite{FoiasJollyLanRupamYangZhang2014}.
\end{rmk}

Now, notice that the result of Proposition \ref{propboundtimederv}, in particular the uniform bound of $\vn$ with respect to the norm in $\mD(A)$ and the uniform bounds of $\rd \vn/\rd t$ with respect to the norms in $V$ and $\mD(A)$, follow from Proposition \ref{apppropboundstildevN} by restricting $\widetilde{\vn}$ to $[0,\infty)$ and choosing, for example,
\[ K_1 = [\rho/\sqrt{2}, \infty)\,,\quad K_2 = [3 \rho/(2 \sqrt{2}), \infty),
\]
so that
\[ r_1 = \dist(K_1,\partial \mB_1) = \rho/2, \quad r_2 = \dist(K_2,\partial \mB_2) = \rho/(2 \sqrt{2}).
\]

\section*{Acknowledgments}
EST would like to thank the \'{E}cole Polytechnique for its kind hospitality, where this work was completed, and the \'{E}cole Polytechnique Foundation for its partial financial support through the 2017-2018 ``Gaspard Monge Visiting Professor" Program. This work is supported in part by the NSF grant number DMS-1516866 and by the ONR grant N00014-15-1-2333. The work of EST was also supported in part by the Einstein Stiftung/Foundation - Berlin, through the Einstein Visiting Fellow Program.

\bibliographystyle{amsplain}
\bibliography{mybib}

\providecommand{\bysame}{\leavevmode\hbox to3em{\hrulefill}\thinspace}
\providecommand{\MR}{\relax\ifhmode\unskip\space\fi MR }
\providecommand{\MRhref}[2]{%
  \href{http://www.ams.org/mathscinet-getitem?mr=#1}{#2}
}
\providecommand{\href}[2]{#2}
\begin{thebibliography}{10}

\bibitem{AlbanezNussenzveigTiti2016}
D.~Albanez, H.~Nussenzveig-Lopes, and E.S. Titi, \emph{Continuous data
  assimilation for the three-dimensional {N}avier--{S}tokes-$\alpha$ model},
  Asymptot. Anal. \textbf{97} (2016), no.~1-2, 139--164.

\bibitem{Altaf2017}
M.~U. Altaf, E.~S. Titi, T.~Gebrael, O.~M. Knio, L.~Zhao, M.~F. McCabe, and
  I.~Hoteit, \emph{Downscaling the 2{D} {B}{\'e}nard convection equations using
  continuous data assimilation}, Comput. Geosci. \textbf{21} (2017), no.~3,
  393--410.

\bibitem{Anthes1974}
R.A. Anthes, \emph{Data assimilation and initialization of hurricane prediction
  models}, Journal of the Atmospheric Sciences \textbf{31} (1974), no.~3,
  702--719.

\bibitem{AschBocquetNodetbook2016}
M.~Asch, M.~Bocquet, and M.~Nodet, \emph{Data assimilation: methods,
  algorithms, and applications}, vol.~11, Society for Industrial and Applied
  Mathematics (SIAM), Philadelphia, PA, 2016.

\bibitem{AzouaniOlsonTiti2014}
A.~Azouani, E.~Olson, and E.S. Titi, \emph{Continuous data assimilation using
  general interpolant observables}, J. Nonlinear Sci. \textbf{24} (2014),
  no.~2, 277--304.

\bibitem{AzouaniTiti2013}
A.~Azouani and E.S. Titi, \emph{Feedback control of nonlinear dissipative
  systems by finite determining parameters -- a reaction-diffusion paradigm},
  Evol. Equ. Control Theory (2014), no.~3, 579--594.

\bibitem{BiswasFoiasMondainiTiti2018}
A.~Biswas, C.~Foias, C.F. Mondaini, and E.S. Titi, \emph{Downscaling data
  assimilation algorithm with applications to statistical solutions of the
  {N}avier-{S}tokes equations}, Annales de l'Institut Henri Poincar\'{e} (C)
  Analyse Non Lin\'{e}aire, (to appear). {arXiv:1711.04067v2} (2018).

\bibitem{BiswasLariosPei2017}
A.~Biswas, J.~Hudson, A.~Larios, and Y.~Pei, \emph{Continuous data assimilation
  for the magnetohydrodynamic equations in 2{D} using one component of the
  velocity and magnetic fields}, arXiv:1704.02082 (2017).

\bibitem{BiswasMartinez2017}
A.~Biswas and V.R. Martinez, \emph{Higher-order synchronization for a data
  assimilation algorithm for the 2{D} {N}avier--{S}tokes equations}, Nonlinear
  Anal. Real World Appl. \textbf{35} (2017), 132--157.

\bibitem{BlocherMartinezOlson2017}
J.~Blocher, V.R. Martinez, and E.~Olson, \emph{Data assimilation using noisy
  time-averaged measurements}, Phys. D (2018),
  https://doi.org/10.1016/j.physd.2017.12.004.

\bibitem{BrezisGallouet1980}
H.~Br{\'e}zis and T.~Gallouet, \emph{Nonlinear {S}chr{\"o}dinger evolution
  equations}, Nonlinear Anal. TMA \textbf{4} (1980), 677--681.

\bibitem{CockburnJonesTiti1997}
B.~Cockburn, D.~Jones, and E.S. Titi, \emph{Estimating the number of asymptotic
  degrees of freedom for nonlinear dissipative systems}, Math. Comp.
  \textbf{66} (1997), no.~219, 1073--1087.

\bibitem{CockburnJonesTiti1995}
B.~Cockburn, D.A. Jones, and E.S. Titi, \emph{Determining degrees of freedom
  for nonlinear dissipative equations}, C. R. Acad. Sci. Paris Sr. I Math.
  \textbf{321} (1995), no.~5, 563--568.

\bibitem{ConstantinFoiasbook}
P.~Constantin and C.~Foias, \emph{Navier--{S}tokes equations}, Chicago Lectures
  in Mathematics, University of Chicago Press, Chicago, IL, 1988.

\bibitem{Daleybook1993}
R.~Daley, \emph{Atmospheric data analysis}, Cambridge Atmospheric and Space
  Science Series, Cambridge University Press, Cambridge, 1993.

\bibitem{DevulderMarionTiti1993}
C.~Devulder, M.~Marion, and E.S. Titi, \emph{On the rate of convergence of the
  nonlinear {G}alerkin methods}, Math. Comp. \textbf{60} (1993), no.~202,
  495--514.

\bibitem{FarhatJohnstonJollyTiti2017}
A.~Farhat, H.~Johnston, M.S. Jolly, and E.S. Titi, \emph{Assimilation of nearly
  turbulent {R}ayleigh-{B}{\'e}nard flow through vorticity or local circulation
  measurements: a computational study}, J. Sci. Comput., (to appear).
  {arXiv:1709.02417 [math.AP]} (2018).

\bibitem{FarhatJollyTiti2015}
A.~Farhat, M.S. Jolly, and E.S. Titi, \emph{Continuous data assimilation for
  the 2{D} {B}{\'e}nard convection through velocity measurements alone}, Phys.
  D \textbf{303} (2015), 59--66.

\bibitem{FarhatLunasinTiti2016a}
A.~Farhat, E.~Lunasin, and E.S. Titi, \emph{Abridged continuous data
  assimilation for the 2{D} {N}avier--{S}tokes equations utilizing measurements
  of only one component of the velocity field}, J. Math. Fluid Mech.
  \textbf{18} (2016), no.~1, 1--23.

\bibitem{FarhatLunasinTiti2016b}
\bysame, \emph{Data assimilation algorithm for 3{D} {B}{\'e}nard convection in
  porous media employing only temperature measurements}, J. Math. Anal. Appl.
  \textbf{438} (2016), no.~1, 492--506.

\bibitem{FarhatLunasinTiti2016c}
\bysame, \emph{On the {C}harney conjecture of data assimilation employing
  temperature measurements alone: The paradigm of 3{D} planetary geostrophic
  model}, Math. Clim. Weather Forecast. \textbf{2} (2016), 61--74.

\bibitem{FarhatLunasinTiti2017}
\bysame, \emph{A data assimilation algorithm: the paradigm of the 3{D}
  {L}eray-alpha model of turbulence}, ``Nonlinear Partial Differential
  Equations Arising from Geometry and Physics", Cambridge University Press,
  London Mathematical Society, Lecture Notes Series (to appear).
  arXiv:1702.01506[math.AP] (2018).

\bibitem{FJKT1988}
C.~Foias, M.S. Jolly, I.G. Kevrekidis, G.R. Sell, and E.S. Titi, \emph{On the
  computation of inertial manifolds}, Phys. Lett. A \textbf{131} (1988),
  no.~7-8, 433--436.

\bibitem{FoiasJollyKevrekidisTiti1990}
C.~Foias, M.S. Jolly, I.G. Kevrekidis, and E.S. Titi, \emph{Dissipativity of
  numerical schemes}, Nonlinearity \textbf{4} (1991), 591--613.

\bibitem{FoiasJollyKevrekidisTiti1994}
\bysame, \emph{On some dissipative fully discrete nonlinear {G}alerkin schemes
  for the {K}uramoto-{S}ivashinsky equation}, Phys. Lett. A \textbf{186}
  (1994), 87--96.

\bibitem{FoiasJollyLanRupamYangZhang2014}
C.~Foias, M.S. Jolly, R.~Lan, R.~Rupam, Y.~Yang, and B.~Zhang, \emph{Time
  analyticity with higher norm estimates for the 2{D} {N}avier--{S}tokes
  equations}, IMA J. Appl. Math. \textbf{80} (2014), no.~3, 766--810.

\bibitem{FoiasManleyTemam1988}
C.~Foias, O.~Manley, and R.~Temam, \emph{Modelling of the interaction of small
  and large eddies in two dimensional turbulent flows}, ESAIM Math. Model.
  Numer. Anal. \textbf{22} (1988), no.~1, 93--118.

\bibitem{FMTT1983}
C.~Foias, O.P. Manley, R.~Temam, and Y.M. Treve, \emph{Asymptotic analysis of
  the {N}avier-{S}tokes equations}, Phys. D \textbf{9} (1983), no.~1-2,
  157--188.

\bibitem{FoiasMondainiTiti2016}
C.~Foias, C.F. Mondaini, and E.S. Titi, \emph{A discrete data assimilation
  scheme for the solutions of the two-dimensional {N}avier--{S}tokes equations
  and their statistics}, SIAM J. Appl. Dyn. Syst. \textbf{15} (2016), no.~4,
  2109--2142.

\bibitem{FoiasProdi1967}
C.~Foias and G.~Prodi, \emph{Sur le comportement global des solutions
  non-stationnaires des {\'e}quations de {N}avier--{S}tokes en dimension 2},
  Rend. Sem. Mat. Univ. Padova \textbf{39} (1967), 1--34.

\bibitem{foiastemam1979}
C.~Foias and R.~Temam, \emph{Some analytic and geometric properties of the
  solutions of the evolution {N}avier-{S}tokes equations}, J. Math. Pures Appl.
  \textbf{58} (1979), no.~3, 339--368.

\bibitem{FoiasTemam1984}
\bysame, \emph{Determination of the solutions of the {N}avier-{S}tokes
  equations by a set of nodal values}, Math. Comp. \textbf{43} (1984), no.~167,
  117--133.

\bibitem{FoiasTiti1991}
C.~Foias and E.S. Titi, \emph{Determining nodes, finite difference schemes and
  inertial manifolds}, Nonlinearity \textbf{4} (1991), no.~1, 135--153.

\bibitem{GarciaArchillaNovoTiti1998}
B.~Garc{\'i}a-Archilla, J.~Novo, and E.S. Titi, \emph{Postprocessing the
  {G}alerkin method: a novel approach to approximate inertial manifolds}, SIAM
  J. Numer. Anal. \textbf{35} (1998), no.~3, 941--972.

\bibitem{GarciaArchillaNovoTiti1999}
\bysame, \emph{An approximate inertial manifolds approach to postprocessing the
  {G}alerkin method for the {N}avier-{S}tokes equations}, Math. Comp.
  \textbf{68} (1999), no.~227, 893--911.

\bibitem{GeshoOlsonTiti2016}
M.~Gesho, E.~Olson, and E.S. Titi, \emph{A computational study of a data
  assimilation algorithm for the two-dimensional {N}avier-{S}tokes equations},
  Commun. Comput. Phys. \textbf{19} (2016), no.~4, 1094--1110.

\bibitem{GottliebToneWang2012}
S.~Gottlieb, F.~Tone, C.~Wang, X.~Wang, and D.~Wirosoetisno, \emph{Long time
  stability of a classical efficient scheme for two-dimensional
  {N}avier-{S}tokes equations}, SIAM J. Numer. Anal. \textbf{50} (2012), no.~1,
  126--150.

\bibitem{GrahamSteenTiti1993}
M.D. Graham, P.H. Steen, and E.S. Titi, \emph{Computational efficiency and
  approximate inertial manifolds for a {B}{\'e}nard convection system}, J.
  Nonlinear Sci. \textbf{3} (1993), no.~1, 153--167.

\bibitem{Guermond2015}
J.-L Guermond and P.~Minev, \emph{High-order time stepping for the
  incompressible {N}avier--{S}tokes equations}, SIAM J. Sci. Comput.
  \textbf{37} (2015), no.~6, A2656--A2681.

\bibitem{GunzburgerManservisi2000}
M.D. Gunzburger and S.~Manservisi, \emph{Analysis and approximation for linear
  feedback control for tracking the velocity in {N}avier–{S}tokes flows},
  Comput. Methods Appl. Mech. Engrg. \textbf{189} (2000), no.~3, 803 -- 823.

\bibitem{Heister2017}
T.~Heister, M.A. Olshanskii, and L.G. Rebholz, \emph{Unconditional long-time
  stability of a velocity--vorticity method for the 2{D} {N}avier--{S}tokes
  equations}, Numer. Math. \textbf{135} (2017), no.~1, 143--167.

\bibitem{HokeAnthes1976}
J.E. Hoke and R.A. Anthes, \emph{The initialization of numerical models by a
  dynamic-initialization technique}, Mon. Wea. Rev. \textbf{104} (1976),
  no.~12, 1551--1556.

\bibitem{JollyKevrekidisTiti1990}
M.S. Jolly, I.G. Kevrekidis, and E.S. Titi, \emph{Approximate inertial
  manifolds for the {K}uramoto-{S}ivashinsky equation: analysis and
  computations}, Phys. D \textbf{44} (1990), no.~1-2, 38--60.

\bibitem{JollyKevrekidisTiti1991}
\bysame, \emph{Preserving dissipation in approximate inertial forms for the
  {K}uramoto-{S}ivashinsky equation}, J. Dynam. Differential Equations
  \textbf{3} (1991), no.~2, 179--197.

\bibitem{JollyMartinezTiti2017}
M.S. Jolly, V.R. Martinez, and E.S. Titi, \emph{A data assimilation algorithm
  for the subcritical surface quasi-geostrophic equation}, Adv. Nonlinear Stud.
  \textbf{17} (2017), no.~1, 167--192.

\bibitem{JonesTiti1992}
D.A. Jones and E.S. Titi, \emph{Determining finite volume elements for the 2{D}
  {N}avier-{S}tokes equations}, Phys. D \textbf{60} (1992), no.~1-4, 165--174.

\bibitem{JonesTiti1993}
\bysame, \emph{Upper bounds on the number of determining modes, nodes, and
  volume elements for the {N}avier-{S}tokes equations}, Indiana Univ. Math. J.
  (1993), 875--887.

\bibitem{Ju2002}
N.~Ju, \emph{On the global stability of a temporal discretization scheme for
  the {N}avier-{S}tokes equations}, IMA J. Numer. Anal. \textbf{22} (2002),
  no.~4, 577--597.

\bibitem{Kalnaybook2003}
E.~Kalnay, \emph{Atmospheric modeling, data assimilation and predictability},
  Cambridge University Press, New York, 2003.

\bibitem{LawStuartZygalakisbook2015}
K.~Law, A.~Stuart, and K.~Zygalakis, \emph{Data assimilation: a mathematical
  introduction}, vol.~62, Springer, 2015.

\bibitem{LunasinTiti2015}
E.~Lunasin and E.S. Titi, \emph{Finite determining parameters feedback control
  for distributed nonlinear dissipative systems -- a computational study},
  Evol. Equ. Control Theory \textbf{6}, no.~4, 535--557.

\bibitem{HarlimMajdabook2012}
A.J. Majda and J.~Harlim, \emph{Filtering complex turbulent systems}, Cambridge
  University Press, Cambridge, 2012.

\bibitem{MargolinTitiWynne2003}
L.G. Margolin, E.S. Titi, and S.~Wynne, \emph{The {P}ostprocessing {G}alerkin
  and {N}onlinear {G}alerkin methods---a truncation analysis point of view},
  SIAM J. Numer. Anal. \textbf{41} (2003), no.~2, 695--714.

\bibitem{MarionTemam1989}
M.~Marion and R.~Temam, \emph{Nonlinear {G}alerkin methods}, SIAM J. Numer.
  Anal. \textbf{26} (1989), no.~5, 1139--1157.

\bibitem{MarkowichTitiTrabelsi2016}
P.A. Markowich, E.S. Titi, and S.~Trabelsi, \emph{Continuous data assimilation
  for the three-dimensional {B}rinkman--{F}orchheimer-extended {D}arcy model},
  Nonlinearity \textbf{29} (2016), no.~4, 1292--1328.

\bibitem{MondainiTiti2018}
C.F. Mondaini and E.S. Titi, \emph{Uniform-in-time error estimates for the
  {P}ostprocessing {G}alerkin method applied to a data assimilation algorithm},
  SIAM J. Numer. Anal. \textbf{56} (2018), no.~1, 78--110.

\bibitem{Nijmeijer2001}
H.~Nijmeijer, \emph{A dynamical control view on synchronization}, Phys. D
  \textbf{154} (2001), no.~3-4, 219--228.

\bibitem{ReichCotterbook2015}
S.~Reich and C.~Cotter, \emph{Probabilistic forecasting and {B}ayesian data
  assimilation}, Cambridge University Press, Cambridge, 2015.

\bibitem{Shen1990}
J.~Shen, \emph{Long time stability and convergence for fully discrete
  {N}onlinear {G}alerkin methods}, Appl. Anal. \textbf{38} (1990), no.~4,
  201--229.

\bibitem{Shen1992b}
J.~Shen, \emph{On error estimates of projection methods for {N}avier–{S}tokes
  equations: First-order schemes}, SIAM J. Numer. Anal. \textbf{29} (1992),
  no.~1, 57--77.

\bibitem{Shen1992}
J.~Shen, \emph{On error estimates of some higher order projection and
  penalty-projection methods for {N}avier-{S}tokes equations}, Numer. Math.
  \textbf{62} (1992), no.~1, 49--73.

\bibitem{Temambook1995}
R.~Temam, \emph{Navier-{S}tokes equations and nonlinear functional analysis},
  CBMS-NSF Regional Conference Series in Applied Mathematics, vol.~66, SIAM,
  Philadelphia, PA, 1995.

\bibitem{TemamDSbook1997}
\bysame, \emph{Infinite-dimensional dynamical systems in mechanics and
  physics}, 2nd ed., {A}pplied {M}athematical {S}ciences, vol.~68,
  Springer-Verlag, New York, 1997.

\bibitem{Temambook2001}
\bysame, \emph{Navier-{S}tokes equations: theory and numerical analysis}, 3rd
  ed., Studies in Mathematics and its Applications, North-Holland Publishing
  Co., Amsterdam, 1984. Reedition in the AMS Chelsea Series, AMS, Providence,
  RI, 2001.

\bibitem{Thau1973}
F.E. Thau, \emph{Observing the state of non-linear dynamic systems}, Int. J.
  Control \textbf{17} (1973), no.~3, 471--479.

\bibitem{Titi1987}
E.S. Titi, \emph{On a criterion for locating stable stationary solutions to the
  {N}avier-{S}tokes equations}, Nonlinear Anal. TMA \textbf{11} (1987), no.~9,
  1085--1102.

\bibitem{Titi1990}
\bysame, \emph{On approximate inertial manifolds to the {N}avier-{S}tokes
  equations}, J. Math. Anal. Appl. \textbf{149} (1990), no.~2, 540--557.

\bibitem{ToneWirosoetisno2006}
F.~Tone and D.~Wirosoetisno, \emph{On the long-time stability of the implicit
  {E}uler scheme for the two-dimensional {N}avier--{S}tokes equations}, SIAM J.
  Numer. Anal. \textbf{44} (2006), no.~1, 29--40.

\end{thebibliography}
\end{document}